\documentclass[a4paper,11pt,reqno]{amsart}

\usepackage{amssymb,amsmath}
\usepackage{todonotes}

\usepackage{graphicx}

\usepackage[cmtip,all]{xy}

\usepackage{csquotes}
\usepackage{fullpage}

\usepackage{hyperref}

\newtheorem{theorem}{Theorem}[section]
\newtheorem{lemma}[theorem]{Lemma}
\newtheorem{prop}[theorem]{Proposition}
\newtheorem{cor}[theorem]{Corollary}

\newtheorem*{thm*}{Main Theorem}

\theoremstyle{definition}
\newtheorem{definition}[theorem]{Definition}

\theoremstyle{remark}
\newtheorem{remark}[theorem]{Remark}

\numberwithin{equation}{section}

\newcommand\bp{\begin{proof}}
\newcommand\ep{\end{proof}}

\newcommand{\Cs}{C$^*$}

\newcommand{\Ws}{W$^*$}

\DeclareMathOperator{\Ext}{Ext}

\DeclareMathOperator{\Rep}{Rep}

\DeclareMathOperator{\Tr}{Tr}

\newcommand{\C}{{\mathbb C}}
\newcommand{\DD}{{\mathbb D}}

\newcommand{\Z}{{\mathbb Z}}

\newcommand{\R}{{\mathbb R}}
\newcommand\T{{\mathbb T}}

\newcommand{\F}{{\mathcal F}}
\newcommand\HH{\mathcal H}

\newcommand{\K}{{\mathcal K}}

\newcommand\OO{\mathcal O}
\newcommand\QQ{\mathcal Q}

\newcommand\TT{\mathcal T}

\newcommand{\Fock}{{\mathcal{F}_{P}}}
\newcommand{\GA}{\tilde{O}_P^+}

\newcommand{\GL}{\mathrm{GL}}

\newcommand\eps{\varepsilon}

\renewcommand{\tilde}{\widetilde}

\makeatletter
\@namedef{subjclassname@2020}{%
  \textup{2020} Mathematics Subject Classification}
\makeatother

\begin{document}


\title[KK-duality for subproduct systems]{KK-duality for the Cuntz--Pimsner algebras of Temperley--Lieb subproduct systems}


\author{Francesca Arici}
\address{Mathematical Institute, Leiden University, P.O. Box 9512, 2300 RA
Leiden, the Netherlands}
\curraddr{}
\email{f.arici@math.leidenuniv.nl}

\author{Dimitris M. Gerontogiannis}
\address{Mathematical Institute of the Polish Academy of Sciences (IMPAN), Jana i Jedrzeja {\'S}niadeckich 8,
00-656, Warsaw, Poland}
\curraddr{}
\email{dgerontogiannis@impan.pl}
\thanks{}

\author{Sergey Neshveyev}
\address{Department of Mathematics, University of Oslo,
P.O. Box 1053 Blindern, NO-0316 Oslo, Norway}
\email{sergeyn@math.uio.no}




\begin{abstract}
We prove that the Cuntz--Pimsner algebra of every Temperley--Lieb subproduct system is $KK$-self-dual. We show also that every such Cuntz--Pimsner algebra has a canonical KMS-state, which we use to construct a Fredholm module representative for the fundamental class of the duality. This allows us to describe the $K$-homology of the Cuntz--Pimsner algebras by explicit Fredholm modules.

Both the construction of the dual class and the proof of duality rely in a crucial way on quantum symmetries of Temperley--Lieb subproduct systems. 
In the simplest case of Arveson's $2$-shift our work establishes $U(2)$-equivariant $KK$-self-duality of~$S^3$. 
\end{abstract}

\maketitle

\section*{Introduction}
The notion of $KK$-duality is a noncommutative analogue of the Spanier--Whitehead duality which relates the homology of a finite complex with the cohomology of some dual finite complex. Namely, two separable C$^*$-algebras $A$ and $B$ are said to be $KK$-dual with dimension shift $i$, if there exists a \emph{fundamental class} $\Delta\in KK_i(A\otimes B,\mathbb C)$ and a \emph{dual class} $\delta \in KK_i(\mathbb C, B\otimes A)$, with Kasparov products satisfying
\begin{equation*}\label{eq:duality_con}
    \delta \otimes_A \Delta = (-1)^i 1_{KK(B,B)}, \qquad \delta \otimes_B \Delta = 1_{KK(A,A)}.
\end{equation*}
The significance of this notion is reflected in natural isomorphisms between the $K$-theory and $K$-homology of the dual C$^*$-algebras obtained by taking Kasparov products with the duality classes, see Section~\ref{sec:duality}. These isomorphisms can be particularly useful in index-theoretic computations, if the duality classes are made explicit. 

As follows from a duality theorem of Kasparov~\cite{Ka88}, for any compact Riemannian manifold~$X$, the C$^*$-algebra $C(X)$ is $KK$-dual to the (graded) algebra of sections of the associated Clifford bundle. This has played an important role in Connes' noncommutative geometry programme~\cite{Con} and is a fundamental ingredient for conceptualizing noncommutative manifolds. The result of Kasparov is also the basis of the Dirac--dual Dirac method for proving the Baum--Connes conjecture, which has led to proofs of $KK$-duality in a variety of cases, see, for example,~\cite{EEK,Em03,EmMe,NP}.

Kaminker and Putnam discovered that hyperbolic dynamics gives another source of examples of $KK$-duality not related to the Dirac--dual Dirac method (at least not in a straightforward way, see the discussion in~\cite[Section~4.5]{KPW}). In~\cite{KaPu} they showed that the Cuntz--Krieger algebras~$\OO_A$ and~$\OO_{A^t}$ are $KK$-dual. More generally, it has been shown~\cite{KPW} that the stable and unstable Ruelle algebras of any irreducible Smale space are $KK$-dual, see \cite{BBGHSW,Ge,GWZ,Nek,PZ,RRS} for applications and related results.

The present work gives a new class of examples of $KK$-dual algebras and can be viewed as a quantum analogue of the result of Kaminker and Putnam.  Specifically, we consider the Cuntz--Pimsner algebras of subproduct systems defined by the so-called Temperley--Lieb polynomials, which were introduced and studied by Habbestad--Neshveyev~\cite{HaNe21,HaNe22} and, in particular cases, earlier by Andersson~\cite{A} and Arici--Kaad~\cite{ArKa21}. An essential feature of such polynomials is that they lead to probably the simplest examples of subproduct systems with large quantum symmetry groups. As a result, the quantum theory provides tools to study these systems that are not accessible from a more classical point of view.

The corresponding Cuntz--Pimsner algebras can be thought of as algebras of functions on algebraic subsets of noncommutative spheres. It can also be shown that they are $KK$-equivalent to Cuntz algebras, see Section~\ref{sec:TL}. 
While this fact already implies abstract $KK$-self-duality by the result of Kaminker and Putnam, our main focus here is on the construction of explicit classes implementing such a duality, including Fredholm module representatives. Only then $KK$-duality is effectively useful, see, for instance,~\cite{BMRS} for a discussion of levels of abstraction in $KK$-duality.

In simple terms, our main result is that for every matrix $A=(a_{ij})_{i,j}\in \GL_m(\C)$, with $m\geq 2$, such that $A\bar A$ is unitary, we consider the noncommutative quadratic polynomials $P$ and $P^\dagger$ defined by $A$ and $A^t$ respectively and show that the Cuntz--Pimsner algebras $\mathcal{O}_P$ and $\mathcal{O}_{P^{\dagger}}$ are $KK$-dual with dimension shift $1$. We remark that the C$^*$-algebras $\OO_P$ and $\OO_{P^\dagger}$ are isomorphic, but not in a canonical way. 




The  Cuntz--Pimsner algebras of subproduct systems~\cite{Vis12} are generalizations of the more familiar Cuntz--Pimsner algebras of C$^*$-correspondences, or of product systems of C$^*$-correspondences, but the reason we say that our result is a quantum analogue of the duality result of Kaminker and Putnam goes deeper than this. As was shown in~\cite{HaNe21}, when $A\bar A=\pm1$, the gauge-invariant part of $\OO_P$ coincides with a noncommutative boundary of an object that can be viewed as a quantum analogue of a tree~\cite{VaVe07}. Conceptually one can then think of $\OO_P$ as a crossed product of that noncommutative boundary by a hyperbolic transformation, so as some sort of quantum Ruelle algebra. With minor modifications, the considerations in~\cite{HaNe21} in fact apply to all  Temperley--Lieb polynomials.

In more precise terms, the appearance of quantum tree-like structures is explained by analyzing quantum symmetries of Temperley--Lieb polynomials. As in \cite{HaNe22}, we denote by $\GA$ the quantum group of unitary transformations leaving a Temperley--Lieb polynomial $P$ invariant up to a phase factor. This is a compact quantum group first studied in~\cite{Mro14}. When $A\bar A=\pm1$, this quantum group decomposes as $O^+_P\times\T$, where $O^+_P$ is the quantum group of unitary transformations leaving $P$ invariant. The quantum group $O^+_P$ is called a free orthogonal quantum group, and its dual discrete quantum group is known to behave in many respects as a free group. Among other properties, it has a noncommutative compactification similar to the end compactification of the Cayley graph of a free group~\cite{VaVe07}. For general $P$, instead of $O^+_P$ one should consider a braided version of free orthogonal quantum groups to get a semidirect product decomposition of $\GA$, but we will not delve deeper into this.

The quantum group $\GA$ plays a central role in the paper. In particular, the candidate dual class $\delta\in KK_1(\mathbb C, \OO_{P^\dagger}\otimes \OO_P)$ is derived from the fusion rules for $\GA$, see Subsections \ref{sec:dual} and \ref{sec:quantum_sym}. Similarly to the Cuntz algebra case, it can be thought to represent the shift along geodesics in a quantum tree. The quantum group $\GA$ is not important for the fundamental class $\Delta\in KK_1(\OO_P\otimes \OO_{P^\dagger},\C)$ though, which is defined as an extension class similarly to~\cite{KaPu}. See Section~\ref{sec:compare} for more on differences and similarities with the Cuntz algebra case.

The bulk of our effort lies in Section \ref{sec:proof}, where we show that the classes we introduce indeed implement a $KK$-duality between $\mathcal{O}_{P}$ and $\mathcal{O}_{P^\dagger}$. To this end, a key observation is that the Kasparov products of those classes lift to $\GA$-equivariant $KK$-theory. This is used in two ways. First, by a result of Mrozinski~\cite{Mro14}, $\GA$ is monoidally equivalent to $U_q(2)$ for a uniquely defined $0<q\le1$. This allows us to reduce the proof of duality to the polynomials $q^{-1/2}X_1X_2-q^{1/2}X_2X_1$. Second, by considering equivariant $KK$-classes we get more tools for manipulating with them, such as Frobenius reciprocity.

It is important to say that, although we do prove a stronger result (Theorem~\ref{thm:main-equi}) than what is needed for our $KK$-duality, it nevertheless does not establish a quantum group equivariant $KK$-duality of~$\OO_P$ and~$\OO_{P^\dagger}$, as defined by Nest and Voigt~\cite{NV}. First, to prove such a duality we would have to work with the Drinfeld double of $\GA$ instead of $\GA$ itself. The second reason is that $\GA$ acts on the right on $\OO_P$, but on the left on $\OO_{P^\dagger}$, and so our classes are not $\GA$-equivariant in the usual sense, only their Kasparov products are. Whether this is a real obstacle or just requires a change in perspective is not clear to us, but we hope to return to this on another occasion. When $P=X_1X_2-X_2X_1$, the quantum group $\GA$ is classical and equals~$U(2)$. In this case, there is no difference between left and right actions and equivariant $KK$-duality makes sense without involving Drinfeld doubles. The corresponding Cuntz--Pimsner algebra~$\OO_P$ is~$C(S^3)$, and our proof implies that $S^3$ (thought of as the unit sphere in $\C^2$) is $U(2)$-equivariantly $KK$-self-dual.

We note that in the proof of $KK$-duality, the fact that the fundamental class is defined abstractly by an extension class, rather than as a concrete $KK_1$-class, does not cause any additional difficulty. On the contrary, it appears to be an asset. Nevertheless, in our case, if one asks for a description of the $K$-homology of the Cuntz--Pimsner algebras suitable for computations, it becomes evident that the fundamental class has to be represented by an explicit odd Fredholm module. In Section \ref{sec:Fredholm}, by adapting an idea of  Goffeng--Mesland \cite{GoMe}, we find such a representative, which also allows us to describe the generators of the $K$-homology groups of~$\mathcal{O}_{P}$ and~$\mathcal{O}_{P^\dagger}$ in terms of $\GA$-equivariant Fredholm modules. A key tool for building the Fredholm module representatives is the unique $\GA$-invariant state on $\OO_P$, which is a canonical KMS-state on $\OO_P$, see Section \ref{sec:KMS}. 

Our work opens several natural directions for further research. One, which we have already mentioned, is to upgrade our $KK$-duality to the Drinfeld double equivariant setting. Another one is to explore whether the proof techniques presented here, which are deeply rooted in the representation theory of quantum groups, can be extended to the broader context of Cuntz--Pimsner algebras of subproduct systems of equivariant correspondences over $KK$-self-dual C$^*$-algebras, a result similar in spirit to the one of \cite{RRS}. This extension would involve addressing additional analytical challenges inherent in the setting of Hilbert modules, a task that we defer to future work.


\smallskip
{\bf Acknowledgement.} The authors are grateful to Erik Habbestad for inspiring discussions at the early stage of the project.

\bigskip

\section{Preliminaries}\label{sec:prelim}
\subsection{Temperley--Lieb subproduct systems}\label{sec:TL}

We start this section by recalling some basic facts from the theory of subproduct systems and their \Cs-algebras. For more details, we refer the reader to \cite{ShSo09,Vis12}. While in their original paper Shalit and Solel studied subproduct systems in the more general setting of \Cs-  and \Ws-correspondences, we will focus here on the Hilbert space case.

By a \emph{subproduct system} of finite-dimensional Hilbert spaces we shall mean a sequence of Hilbert spaces $\mathcal{H} = \{H_n\}_{n \in \Z_+}$, with $\dim H_0=1$ and $\dim H_1<\infty$, together with isometries 
$$
w_{m,n}\colon H_{m+n} \to H_m \otimes H_n
$$ 
satisfying
\[ (w_{m,n}\otimes \iota) w_{m+n,k} = (\iota\otimes w_{n,k}) w_{m,n+k} \colon H_{m+n+k} \to H_m \otimes H_n \otimes H_k\]
for all $m,n,k \in \Z_+$, where $\iota$ denotes the identity operator. The Fock space associated to $\mathcal{H}$ is the direct sum Hilbert space 
$$
\mathcal{F}_\mathcal{H} := \bigoplus_{n\ge0} H_n.
$$

Define creation operators on $\mathcal{F}_{\mathcal{H}}$ by
\[ S_\xi(\zeta) := w_{1,n}^*(\xi\otimes \zeta), \quad \xi \in H_1,\ \zeta \in H_n.\]
The \emph{Toeplitz algebra} associated to $\mathcal{H}$ is the unital \Cs-algebra generated by $S_1,S_2,...,S_m$, where $S_i = S_{\xi_i}$ for an orthonormal basis $(\xi_i)_{i=1}^m$ of $H_1$.  It is straightforward to verify that $1_{\mathcal{F}_\HH} - \sum_i S_iS_i^*$ is the rank-one projection onto $H_0$, and it follows that the compacts $\K(\mathcal{F}_\mathcal{H})$ are contained in~$\mathcal{T}_{\mathcal{H}}$ (see \cite[Corollary 3.2]{Vis12}). The associated \emph{Cuntz--Pimsner algebra} is defined as the quotient in the extension
\begin{equation}\label{eq:CPext}
\xymatrix{0 \ar[r] & \mathcal{K}(\mathcal{F}_\HH) \ar[r] & \mathcal{T}_\HH \ar[r] & \mathcal{O}_\HH \ar[r] & 0}.
\end{equation}%


A subproduct system is called \emph{standard} if $H_0=\C$, $H_{m+n}  \subseteq H_{m} \otimes H_{n}$ and $w_{m,n}$ are the embedding maps. In this case we have
\[ \mathcal{F}_{\mathcal{H}} \subseteq \bigoplus_{n\ge0} H_1^{\otimes n} \quad \mbox{ and } \quad S_\xi(\zeta) = f_{n+1}(\xi\otimes \zeta), \quad \xi \in H_1, \zeta \in H_n, \]
where $f_{n+1}$ is the projection $H_1^{\otimes (n+1)} \to H_{n+1}$.

As pointed out in \cite{ShSo09}, standard subproduct systems of finite-dimensional Hilbert spaces provide a natural framework for studying row-contractive tuples of operators subject to polynomial constraints, as made transparent by the existence of a \emph{noncommutative Nullstellensatz}.

\begin{prop}[{\cite[Proposition 7.2]{ShSo09}}]\label{prop:1to1poly}
Let $H$ be an $m$-dimensional Hilbert space with an orthonormal basis $\lbrace \xi_i \rbrace_{i=1}^m$. Then there is a bijective inclusion-reversing correspondence between the proper homogeneous ideals $J  \subset \mathbb{C} \langle X_1,\ldots,X_m \rangle$ and the standard subproduct systems $\lbrace H_n \rbrace_{n \in \mathbb{Z}_{+}}$ with $H_1 \subseteq H$.
\end{prop}

The correspondence works as follows: for a noncommutative polynomial $P=\sum c_{\alpha} X_{\alpha}$ in variables $X_1\dots,X_m$, where $\alpha=(\alpha_1,\dots,\alpha_n)$, $\alpha_i\in\{1,\dots,m\}$ and $X_\alpha=X_{\alpha_1}\dots X_{\alpha_n}$, we write $P(\xi) := \sum c_{\alpha} \xi_{\alpha_1} \otimes \dotsc \otimes \xi_{\alpha_n}$. To any proper homogeneous ideal $J \subset \langle X_1,\ldots,X_m \rangle$, we associate the standard subproduct system with fibres $H_n:= H^{\otimes n} \ominus \lbrace P(\xi) : P \in J^{(n)} \rbrace$, for every $n \geq 0$, where $J^{(n)}$ denotes the degree $n$ component of the ideal $J$. 




Note that, while the subproduct system associated to a proper homogeneous ideal $J  \subset \mathbb{C} \langle X_1,\ldots,X_m \rangle$ depends on the choice of an orthonormal basis for the Hilbert space $H$, different choices give rise to isomorphic subproduct systems. In a basis-independent form Proposition~\ref{prop:1to1poly} can also be formulated as follows: there is a bijective inclusion-reversing correspondence between the proper homogeneous ideals $J$ of the tensor algebra $T(H)$ and the standard subproduct systems $\lbrace H_n \rbrace_{n \in \mathbb{Z}_{+}}$ with $H_1 \subseteq H$.

\smallskip

In the present paper we will focus on a class of standard subproduct systems induced by special quadratic polynomials, called \emph{Temperley--Lieb} polynomials, introduced in \cite{HaNe21} and further studied in \cite{HaNe22}.

\begin{definition}[\cite{HaNe21}]\label{def:TP_pol}
Let $H$ be a finite-dimensional Hilbert space of dimension $m \geq 2$. A nonzero vector $P \in H\otimes H$ is called \emph{Temperley--Lieb} if there is $\lambda > 0$ such that the orthogonal projection $e\colon H\otimes H \to \C P$ satisfies
\[ (e\otimes 1)(1\otimes e)(e\otimes 1) = \dfrac{1}{\lambda}(e\otimes 1)\quad\text{in}\quad B(H\otimes H\otimes H).\]
\end{definition}

The standard subproduct system $\mathcal{H}_P = \{H_n\}_{n\in\Z_+}$ defined by the ideal $\langle P\rangle\subset T(H)$ generated by $P$ is called a Temperley--Lieb subproduct system. We write $\mathcal{F}_P = \mathcal{F}_{\mathcal{H}_P}$, $\mathcal{T}_P = \mathcal{T}_{\mathcal{H}_P}$ and $\mathcal{O}_P = \mathcal{O}_{\mathcal{H}_P}$.

We will often fix an orthonormal basis in $H$ and identify $H^{\otimes n}$ with the space of homogeneous noncommutative polynomials of degree $n$ in variables $X_1,\dots,X_m$. In particular, we write a vector $P\in H\otimes H$ as a noncommutative polynomial $P=\sum^m_{i,j=1}a_{ij}X_iX_j$. Consider the matrix $A=(a_{ij})_{i,j}$. By \cite[Lemma~1.4]{HaNe21}, $P$ is Temperley--Lieb if and only if the matrix $A\bar A$ is unitary up to a (nonzero) scalar factor, where $\bar A=(\bar a_{ij})_{i,j}$. Since the ideal generated by $P$ does not change if we multiply $P$ by a nonzero factor, we may always assume that $A\bar A$ is unitary.


The following theorem describes a complete set of relations in $\TT_P$. Denote by $\gamma$ the shift to the left on $c := C(\Z_+ \cup \{\infty\})$, so that $\gamma(f)(n)=f(n+1)$.

\begin{theorem}[{\cite[Theorem~2.11]{HaNe22}}] \label{thm: univ_relations}
Let $A=(a_{ij})_{i,j}\in \GL_m(\C)$ ($m\geq 2$) be such that $A\bar A$ is unitary. Let $q\in(0,1]$ be the number such that $\Tr(A^*A)=q+q^{-1}$. Consider the noncommutative polynomial $P=\sum^m_{i,j=1}a_{ij}X_iX_j$.  Then $\mathcal{T}_P$ is a universal C$^*$-algebra generated by the \Cs-algebra $c= C(\Z_+ \cup \{\infty\})$ and elements $S_1,S_2,...,S_m$ satisfying the relations
\[ fS_i = S_i\gamma(f) \quad (f \in c,\ 1 \leq i \leq m), \quad \sum^m_{i=1} S_iS_i^* = 1 - e_0, \quad \sum^m_{i,j=1} a_{ij}S_iS_j = 0,  \]
\[S_i^*S_j + \phi\sum^m_{k,l=1}a_{ik}\bar{a}_{jl}S_kS_l^* = \delta_{ij}1 \quad (1 \leq i,j \leq m), \]
where $e_0\in c$ is the characteristic function of $\{0\}$ and $\phi\in c$ is the element given by
\begin{equation} \label{eq:phi(n)}
\phi(n)= \dfrac{[n]_q}{[n+1]_q},\quad\text{with}\quad [n]_q=\begin{cases} \dfrac{q^{n}-q^{-n}}{q - q^{-1}}, &\text{if}\ q<1,\\
n,&\text{if}\ q=1.\end{cases}
\end{equation}
\end{theorem}

Here $c$ is identified with a unital subalgebra of $\K(\mathcal{F}_P)+\C1 \subset \mathcal{T}_P$, with $e_n\in c$ being identified with the projection $\mathcal{F}_P \to H_n$. Note also that $\phi(n)\to q$ as $n\to+\infty$.

The relations become slightly simpler if we write $P$ in a standard form. Namely, by \cite[Proposition~1.5]{HaNe21}, up to a unitary change of variables and rescaling, we may assume that our Temperley--Lieb polynomial $P$ has the form 
\begin{equation}
    \label{eq:TLPoly_normal} P = \sum_{i=1}^{m} a_i X_iX_{m-i+1},    \quad \text{with}\quad \vert a_i a_{m-i+1} \vert =1.
\end{equation} 

Next, let us recall what is known about the $K$-theory of $\TT_P$ and $\OO_P$.

\begin{theorem}[{\cite[Theorem~3.1 and Corollary~4.4]{HaNe22}}]
For every Temperley–Lieb polynomial~$P$, the embedding map $\mathbb{C} \to \mathcal{T}_P$ is a $KK$-equivalence. Moreover, we have $[e_0]=(2-m)[1]$ in~$K_0(\TT_P)$.
\end{theorem}

As a consequence, the six-term exact sequence induced by the defining extension \eqref{eq:CPext} simplifies notably, and one obtains the following result about the $K$-theory of the Cuntz--Pimsner algebra of a Temperley--Lieb subproduct system. 

\begin{cor}[{\cite[Corollary~4.4]{HaNe22}}]\label{cor:Ktheory}
For every Temperley--Lieb polynomial $P$ in $m$ variables, we have 
    \[K_0(\mathcal{O}_{P}) \cong \mathbb{Z}/(m-2)\mathbb{Z}, \qquad K_1(\mathcal{O}_{P}) \cong \begin{cases}
        \mathbb{Z}, & m=2,\\
        0, & m \geq 3.
    \end{cases}\]
\end{cor}

By {\cite[Corollary~2.9]{HaNe22}}, the algebras $ \mathcal{T}_P$ and  $\mathcal{O}_P$ are nuclear for any Temperley--Lieb polynomial. As a consequence, the Toeplitz extension \eqref{eq:CPext}
is semi-split, and hence it induces a six-term exact sequence in $K$-homology. Using the $KK$-equivalence result above, together with similar arguments to those in the proof of \cite[Corollary~ 4.4]{HaNe22}, we obtain the following computation for the $K$-homology groups of the Cuntz--Pimsner algebra:
\begin{cor}\label{cor:Khomology}
    For every Temperley--Lieb polynomial $P$ in $m$ variables, we have 
    \[K^1(\mathcal{O}_{P}) \cong \mathbb{Z}/(m-2)\mathbb{Z}, \qquad K^0(\mathcal{O}_{P}) \cong \begin{cases}
        \mathbb{Z}, & m=2,\\
        0, & m \geq 3.
    \end{cases}\]
\end{cor}

Let us finally remark that the \Cs-algebras $\mathcal{O}_{P}$  are in the UCT class by a two-out-of-three argument (see, e.g., \cite[Proposition~8.8]{Kat04}), since $\mathcal{K}(\mathcal{F}_{P})$ is in the UCT class and $\mathcal{T}_{P}$ is $KK$-equivalent to $\mathbb{C}$. Corollary~\ref{cor:Ktheory} implies then that $\OO_P$ is $KK$-equivalent to $C(\T)$ for $m=2$ and to $\OO_{m-1}$ for $m\ge3$.

\subsection{KK-duality}\label{sec:duality}
$KK$-duality is a noncommutative analogue of Spanier--Whitehead duality for finite complexes. In the literature it is sometimes called $K$-theoretic Poincar\'e duality. We define it following \cite{Em03, KaPu, KPW}, but with a small modification in order to avoid a repeated use of flip maps, which are not that innocent when we start to take into account quantum symmetries. 
\begin{definition}\label{def:KK-duality}
Let $A$ and $B$ be separable \Cs-algebras. 
We say that $A$ and $B$ are \emph{KK-dual with dimension shift $i\in \{0,1\}$}, if there is a $K$-homology class $\Delta \in KK_i(A\otimes B,\mathbb{C})$ (\emph{fundamental class}) and a $K$-theory class $\delta \in KK_i(\mathbb{C}, B \otimes A) $ (\emph{dual class}) such that
\[\delta \otimes_{A} \Delta = (-1)^i 1_{KK(B,B)}, \qquad \delta \otimes_{B} \Delta = 1_{KK(A,A)}.\]
\end{definition}
The Kasparov products here are understood as 
\begin{align*}
\delta \otimes_{A} \Delta &:= (\delta \otimes 1_{KK(B,B)})\otimes_{B\otimes A\otimes B}(1_{KK(B,B)}\otimes \Delta),\\
\delta \otimes_{B} \Delta &:= (1_{KK(A,A)}\otimes \delta)\otimes_{A\otimes B\otimes A}(\Delta \otimes 1_{KK(A,A)}).
\end{align*}
We note that a duality pair $(\Delta,\delta)$ as per Definition \ref{def:KK-duality} gives a duality pair $(\Delta,\sigma_*(\delta))$ in the usual way, and vice versa. Here $\sigma\colon B\otimes A\to A\otimes B$ is the flip map. It is known that $\delta$ is unique for $\Delta$ and vice versa, while the collection of duality pairs $(\Delta,\delta)$, when it is nonempty, essentially corresponds to the invertible elements in the ring $KK(A,A),$ see \cite[Corollary 2.9]{BMRS}. Further, one obtains isomorphisms
\begin{equation}\label{eq:slantprod}
- \otimes_A \Delta\colon  K_j(A)\to K^{j+i}(B), \qquad
\sigma_*(\delta) \otimes_B -\colon K^j(B)\to K_{j+i}(A).
\end{equation}

A separable \Cs-algebra satisfying the UCT has a $KK$-dual if and only if its $K$-theory is finitely generated, see \cite[Proposition 5.9]{EmMe}. It follows that $\OO_P$ has a $KK$-dual, and in view of the result of Kaminker and Putnam~\cite{KaPu} it is natural to conjecture that such a dual has the form $\OO_{P^\dagger}$ for a suitable $P^\dagger$. 

\bigskip

\section{Duality classes for Temperley--Lieb subproduct systems}
Consider the involutive operation $\dagger$ on the set of Temperley--Lieb polynomials defined by transposition of matrices, so if $P=\sum_{i,j}a_{ij}X_iX_j$, then $P^\dagger=\sum_{i,j}a_{ji}X_iX_j$. In a basis-independent form this means that given a Temperley--Lieb vector $P\in H\otimes H$ we consider the Temperley--Lieb vector $P^\dagger=P_{21}$ obtained by flipping the factors.

Our goal is to show that $\OO_P$ and $\OO_{P^\dagger}$ are $KK$-dual to each other. Note that the C$^*$-algebras~$\OO_P$ and~$\OO_{P^\dagger}$ are isomorphic, since in view of the standard form \eqref{eq:TLPoly_normal} we can find a unitary matrix $v$ such that $A^t=vAv^t$. Since there is no preferred choice for such a unitary $v$, the isomorphism $\OO_P\cong\OO_{P^\dagger}$ is noncanonical.

In order to simplify some of the formulas we consider the Temperley--Lieb polynomials in the form \eqref{eq:TLPoly_normal}, so for the rest of this section we assume that
$$
P = \sum_{i=1}^{m} a_i X_iX_{m-i+1},    \quad \text{with}\quad \vert a_i a_{m-i+1} \vert =1,
$$
and therefore $P^{\dagger} = \sum_{i=1}^{m} {a}_{m-i+1} X_iX_{m-i+1}$. Let $0<q\le1$ be such that
\begin{equation}\label{eq:q}
q+q^{-1}=\sum^m_{i=1}|a_i|^2.
\end{equation}

We denote by $S_i$ and $T_i$ the generators of~$\mathcal{T}_P$ and~$\mathcal{T}_{P^\dagger}$, respectively, and we use lower-case letters to denote their images in the Cuntz--Pimsner algebras, that is, we shall denote the generators of $\mathcal{O}_P$ by $s_i$ and the generators of $\mathcal{O}_{P^\dagger}$ by $t_i$. 

\subsection{The fundamental class}\label{sec:fundamental}
To construct the fundamental class we use that both~$\mathcal{T}_P$ and~$\mathcal{T}_{P^\dagger}$ can be faithfully represented on the same Fock space $\F_P$. Namely, consider the left and right creation operators on $\F_P$ defined by
$$
L_i\xi=f_{n+1}(\xi_i\otimes\xi),\qquad R_i\xi=f_{n+1}(\xi\otimes\xi_i)
$$
for $\xi\in H_n$ and an orthonormal basis $(\xi_i)_{i=1}^{m}$ in $H_1=H=\C^m$. Thus, $L_i$ are the generators $S_i$ of $\TT_P$, but since we now work with two Toeplitz algebras, we want to distinguish between these generators and concrete operators on $\F_P$.

We have a unitary isomorphism $\F_P\cong\F_{P^\dagger}$ induced by the maps $H^{\otimes n}\to H^{\otimes n}$,
$$
\zeta_1\otimes\zeta_2\otimes\dots\otimes\zeta_n\mapsto\zeta_n\otimes\dots\otimes\zeta_2\otimes\zeta_1.
$$
This unitary isomorphism transfers the generators $T_i$ of $\TT_{P^\dagger}$ into the operators $R_i$, so we have a faithful representation
$$
\TT_{P^\dagger}\to B(\F_P),\quad T_i\mapsto R_i.
$$

Now, the fundamental class is built from the $*$-homomorphisms $\tau_P\colon\mathcal{O}_P\to \mathcal{Q}(\Fock), \tau_{P^\dagger}\colon\mathcal{O}_{P^\dagger}\to \mathcal{Q}(\Fock)$ into the Calkin algebra $\mathcal{Q}(\Fock)$ given by 
\begin{equation}\label{eq:Delta}
\tau_P(s_i):=L_i + \mathcal{K}(\Fock),\qquad \tau_{P^\dagger}(t_i):=R_i + \mathcal{K}(\Fock).
\end{equation}
The images of these maps commute due to the following result, which is a combination of Lemma 4.3, Corollary 4.4 and Remark 4.5 of \cite{HaNe21}.
\begin{lemma}[{\cite{HaNe21}}]\label{lem:norm_est}
\label{lem:comm_est}
For every $1\leq i,j\leq m$, we have $[L_i,R_j]=0$. Further, if $q<1$, then we can find a constant $C>0$ depending only on $q$ so that for every $n\geq 0$, $$\|[L_i^*,R_j]|_{H_n}\|\leq Cq^n.$$ On the other hand, if $q=1$, there is a constant $C'>0$ such that $\|[L_i^*,R_j]|_{H_n}\|\leq C'n^{-1/2},$ for every $n\geq 0$. Consequently, it holds that $[L,R]\in \mathcal{K}(\Fock)$, for every $L\in \TT_P$ and $R\in \TT_{P^\dagger}.$
\end{lemma}

It should be said that formally the estimates in \cite{HaNe21} are proved for a particular class of Temperley--Lieb polynomials. However, the key estimate (\cite[Lemma 4.3]{HaNe21}) can be interpreted as a result about the Temperley--Lieb algebras and as such remains true for arbitrary Temperley--Lieb polynomials.

As a result, since $\mathcal{O}_P$ and $\mathcal{O}_{P^\dagger}$ are nuclear, we can multiply $\tau_P\colon\mathcal{O}_P\to \mathcal{Q}(\Fock), \tau_{P^\dagger}\colon\mathcal{O}_{P^\dagger}\to \mathcal{Q}(\Fock)$ to obtain a $*$-homomorphism 
\begin{equation}\label{eq:fundext}
\tau\colon \mathcal{O}_{P}\otimes \mathcal{O}_{P^\dagger} \to \mathcal{Q}(\Fock)
\end{equation}
which yields an extension class $[\tau]\in \Ext(\mathcal{O}_{P}\otimes \mathcal{O}_{P^\dagger},\C)\cong KK_1(\mathcal{O}_{P}\otimes \mathcal{O}_{P^\dagger},\mathbb C).$

\begin{definition}
\label{def:fundCl}
The \emph{fundamental class} $\Delta$ is the image of $[\tau]$ in $KK_1(\mathcal{O}_{P}\otimes \mathcal{O}_{P^\dagger},\mathbb C).$ 
\end{definition}

In Section \ref{sec:Fredholm}, for $q<1$, we will construct a Fredholm module representative of $\Delta$ using KMS-states for $\mathcal{O}_P$ and $\mathcal{O}_{P^\dagger}$. 

\subsection{The dual class}\label{sec:dual}
We aim to construct a $*$-homomorphism $C_0(\mathbb R)\to M_2(\mathcal{O}_{P^\dagger} \otimes \mathcal{O}_{P})$ which under Bott periodicity yields a $K$-theory duality class in $KK_1(\mathbb C, \mathcal{O}_{P^\dagger}\otimes \mathcal{O}_{P}).$ To this end, we consider the following operator
\begin{equation}\label{eq:w}
w : = \sum_{i=1}^m \begin{pmatrix} t_i^* \otimes s_i &  q^{1/2} a_i t_i \otimes  s_{m-i+1} \\
q^{1/2} \overline{a}_i t_{i}^* \otimes s_{m-i+1}^* & q a_i \overline{a}_{m-i+1} t_{i} \otimes s_{i}^*
\end{pmatrix} \in M_2(\mathcal{O}_{P^\dagger}\otimes \mathcal{O}_{P}).
\end{equation}
\begin{lemma}\label{lem:unitary}
The operator $w$ is unitary.
\end{lemma}

This can be checked by a direct computation using the relations in $\OO_P$ and $\OO_{P^\dagger}$, but later we will obtain a stronger result with a more conceptual proof.

\begin{definition}\label{def:dualCl}
The \emph{dual class} $\delta\in KK_1(\mathbb C, \mathcal{O}_{P^\dagger}\otimes \mathcal{O}_{P})$ is defined as $\beta \otimes_{C_0(\mathbb R)} [\overline{w}],$ where $\beta\in KK_1(\mathbb C, C_0(\mathbb R))$ is the Bott class and $[\overline{w}]\in KK(C_0(\mathbb R), \mathcal{O}_{P^\dagger}\otimes \mathcal{O}_{P})$ is given by the $*$-homomorphism $\overline{w}\colon  C_0(\mathbb R)\to M_2(\mathcal{O}_{P^\dagger} \otimes \mathcal{O}_{P}),$ 
$$
\overline{w}(z-1_{C(\mathbb T)})=w^*-1,
$$ 
and where $z-1_{C(\mathbb T)}$ is viewed as a function in $C(\mathbb T)$ that generates a copy of $C_0(\mathbb R)$.
\end{definition}

Our convention for the Bott class is that it is inverse to the class in $KK_1(C_0(\R),\C)\cong \mathrm{Ext}(C_0(\R),\C)$ obtained by restricting the usual Toeplitz extension
\[
\xymatrix{
0 \ar[r] & \K \ar[r] & \TT \ar[r] &C(\T) \ar[r] &0}
\]
to $C_0(\R)\subset C(\T)$.

From now on, when we want to define a $*$-homomorphism of $C_0(\R)$, we will usually write its unital extension to $C(\T)$ by specifying the image of the generator $z\in C(\T)$.

\subsection{An extension representing the Kasparov product}

Consider the element 
\[
\beta^{-1}\otimes\delta\otimes_{\OO_{P^\dagger}}\Delta\in KK_1(C_0(\R)\otimes\OO_P,\OO_P)\cong \mathrm{Ext}(C_0(\R)\otimes\OO_P,\OO_P).
\]
It is not difficult to see that by definition it is represented by the extension with Busby invariant
\[
C_0(\R)\otimes\OO_P\to M_2(\QQ(\F_P\otimes\OO_P))\cong \QQ(\F_P^2\otimes\OO_P),
\]
\begin{equation}\label{eq:ext}
z\otimes1\mapsto (\tau_{P^\dagger}\otimes\iota)(w^*),\quad 1\otimes s_i\mapsto\begin{pmatrix}
    \tau_P(s_i)\otimes 1 & 0\\ 0 & \tau_P(s_i)\otimes1
\end{pmatrix},
\end{equation}
where $\F_P\otimes\OO_P$ denotes the right C$^*$-Hilbert $\OO_P$-module defined as the exterior tensor product of the C$^*$-Hilbert modules~$\F_P$ and~$\OO_P$ over $\C$ and $\OO_P$, respectively, and we use that $\QQ(\F_P)\otimes\OO_P$ can be viewed as a subalgebra of $\QQ(\F_P\otimes\OO_P)$. We need to show that the class of this extension equals $\beta^{-1}\otimes1_{KK(\OO_P,\OO_P)}$, that is, it coincides with the class of the extension with Busby invariant
\[
C_0(\R)\otimes\OO_P\to\QQ(\ell^2(\Z_+)\otimes\OO_P),
\]
\begin{equation} \label{eq:what-we-want}
z\otimes 1\mapsto u\otimes 1,\quad 1\otimes s_i\mapsto 1\otimes s_i,
\end{equation}
where $u\colon \ell^2(\Z_+)\to\ell^2(\Z_+)$ is the shift to the right. We will do this in the next section using
indirect methods.

\subsection{Comparison with the Cuntz algebra case}\label{sec:compare}

In this subsection we compare our construction and the problem of identifying the extension classes with the Cuntz--Krieger algebra case studied by Kaminker and Putnam~\cite{KaPu}. This will not play any role in the subsequent considerations.

To simplify matters we will only consider the Cuntz algebras $\OO_m$, $m\ge2$. Denote by $s_i$ the generators of $\OO_m$ and by $S_i$ the generators of the Cuntz--Toeplitz algebra $\TT_m$. Consider also the full Fock space $\F_m:=\C\oplus\C^m\oplus(\C^m\otimes\C^m)\oplus\dots$ and the left and right creation operators $L_i$ and $R_i$ on $\F_m$. These operators define homomorphisms $\tau_m,\tau_m^\dagger\colon \OO_m\to \QQ(\F_m)$ with commuting images.

Our Definitions~\ref{def:fundCl} and~\ref{def:dualCl} are motivated by the duality classes $\delta$ and $\Delta$ for $\OO_m$ introduced in~\cite{KaPu}, the main difference is that Kaminker and Putnam use the unitary
$$
w_m:=\sum^m_{i=1}s_i^*\otimes s_i\in\OO_m\otimes\OO_m
$$
to define $\delta$. To prove that $(1_{KK(\OO_m,\OO_m)}\otimes\delta) \otimes_{\OO_m\otimes\OO_m\otimes\OO_m} (\Delta\otimes1_{KK(\OO_m,\OO_m)}) = 1_{KK(\OO_m,\OO_m)}$ one has to show that the class of the extension with Busby invariant
$$
C_0(\R)\otimes\OO_m\to \QQ(\F_m\otimes\OO_m),
$$
\begin{equation}\label{eq:ext-KP}
z\otimes1\mapsto (\tau_m^\dagger\otimes\iota)(w^*_m),\quad 1\otimes s_i\mapsto \tau_m(s_i)\otimes 1,
\end{equation}
equals $\beta^{-1}\otimes1_{KK(\OO_m,\OO_m)}$. This can be done along the following lines, cf.~\cite{KaPu,KPW}.

First, we have a unitary isomorphism $U\colon \F_m\otimes\OO_m\to \ell^2(\Z_+)\otimes\OO_m$ of right C$^*$-Hilbert $\OO_m$-modules defined by
$$
U(\xi_{i_1}\otimes\dots\otimes\xi_{i_n}\otimes a):=\delta_n\otimes s_{i_1}\dots s_{i_n}a.
$$
The unitary $(\tau_m^\dagger\otimes\iota)(w_m)\in\QQ(\F_m\otimes\OO_m)$ lifts to the coisometry
$$
W_m:=\sum^m_{i=1}R_i^*\otimes s_i\in M(\K(\F_m)\otimes\OO_m),
$$
and then $U$ intertwines $W_m^*$ with $u\otimes 1\in M(\K(\ell^2(\Z_+))\otimes\OO_m)$.

The unitary $U$ does not intertwine $\tau_m(s_i)\otimes 1$ with $1\otimes s_i$, however. In order to deal with this, consider the automorphism $\theta$ of $C(\T)\otimes\OO_m$ defined~by
$$
\theta(z\otimes 1):=z\otimes1,\quad \theta(1\otimes s_i):=\bar z\otimes s_i.
$$
It is not difficult to show that the restriction of $\theta$ to $C_0(\R)\otimes\OO_m$ is homotopic to the identity automorphism, see the proofs of \cite[Theorem~5.1]{PZ} or \cite[Lemma~7.2]{KPW}. It follows that replacing homomorphism~\eqref{eq:ext-KP} by its composition with~$\theta$ does not change the class in $KK_1(C_0(\R)\otimes\OO_m,\OO_m)$. In other words, instead of~\eqref{eq:ext-KP} we can consider the homomorphism
$$
z\otimes1\mapsto (\tau_m^\dagger\otimes\iota)(w^*_m),\quad 1\otimes s_i\mapsto \sum^m_{j=1}\tau_m^\dagger(s_j^*)\tau_m(s_i)\otimes s_j.
$$
Now, a simple computation shows that $U$ intertwines the lift $\sum_j R_j^*L_i\otimes s_j$ of $\sum_j \tau_m^\dagger(s_j^*)\tau_m(s_i)\otimes s_j$ with $1\otimes s_i$, and we are done.

\smallskip

Returning to the Temperley--Lieb polynomials, one can try to apply a similar strategy. As we will see in the next section, the unitary $(\tau_{P^\dagger}\otimes\iota)(w)$ lifts to a coisometry with the source projection $1-\begin{pmatrix}
    e_0\otimes 1 & 0 \\ 0 & 0
\end{pmatrix}$. Using this we can define an isometry $\ell^2(\Z_+)\otimes\OO_P\to\F_P^2\otimes\OO_P$ intertwining $u\otimes 1$ with the lift of $(\tau_{P^\dagger}\otimes\iota)(w^*)$. However, this isometry is far from being unitary. We also do not see an easy homotopy/perturbation argument that would allow us to compare the homomorphisms of $\OO_P$ into $\QQ(\ell^2(\Z_+)\otimes\OO_P)$ and $\QQ(\F_P^2\otimes\OO_P)$ by means of this isometry, despite the fact that an analogue of the automorphism~$\theta$ still makes sense. 

\bigskip

\section{Proof of duality}\label{sec:proof}

\subsection{Quantum symmetries of Temperley--Lieb polynomials}\label{sec:quantum_sym}

For every Temperley--Lieb polynomial $P=\sum_{i,j}a_{ij}X_iX_j$ we consider the compact quantum group $\GA$ of unitary transformations leaving $P$ invariant up to a phase factor. This quantum group was introduced by Mrozinski in~\cite{Mro14}. As in~\cite{HaNe22}, we rephrase its definition as follows.

\begin{definition} \label{def:GA}
We define the algebra $\C[\GA]$ of regular functions on $\GA$ as the universal unital $*$-algebra generated by a unitary element~$d$ and elements $v_{ij}$, $1 \leq i,j \leq m$, such that
\[V = (v_{ij})_{i,j}\ \ \text{is unitary and}\ \ VAV^t = dA.\]
This is a Hopf $*$-algebra with comultiplication defined on generators as
\[ \Delta(d) = d\otimes d, \quad \Delta(v_{ij}) = \sum_{k} v_{ik}\otimes v_{kj}.\]
\end{definition}

We will work only with reduced forms of quantum groups, so by $C(\GA)$ we mean the norm completion of $\C[\GA]$ in the GNS-representation defined by the Haar state. Generally, our notation and conventions for quantum groups follow~\cite{HaNe22}, for more information on the subject the reader can consult~\cite{NT}.

A unitary representation of $\GA$ on a Hilbert space $H'$ is a unitary $U\in M(\K(H')\otimes C(\GA))$ such that
$(\iota\otimes\Delta)(U)=U_{12}U_{13}$. Therefore $V$ is a unitary representation of $\GA$ on $\C^m$, while $d$ is a one-dimensional unitary representation. By definition we have
\begin{equation*} \label{eq:poly-symmetry}
V_{13}V_{23}(P\otimes 1) = P\otimes d \quad \mbox{in}\quad \C^m \otimes\C^m \otimes \C[\GA].
\end{equation*}
In other words, the embedding $\C\to\C^m\otimes\C^m$, $1\mapsto P$, intertwines the unitary representations~$d$ and $V\otimes V$ of $\GA$, which is the precise meaning of that the action of $\GA$ rescales $P$ by the character $d$.

For every $n\ge0$, the Hilbert space $H_n=f_n(\C^m)^{\otimes n}$ carries an irreducible representation $U_n$ of $\GA$ obtained by restriction from $V^{\otimes n}$. We will write $H_n[1]$ for the Hilbert space $H_n$ viewed as the underlying space of $U_n\otimes d$. In a similar way we will write $[1]H_n$ for the Hilbert space~$H_n$ viewed as the underlying space of $d\otimes U_n$. Then the fusion rules for $\GA$ are $H_{n}[1]\cong [1]H_{n}$,
\begin{equation}\label{eq:fusion}
H_n\otimes H_1\cong H_{n+1}\oplus H_{n-1}[1],\qquad H_1\otimes H_n\cong H_{n+1}\oplus [1]H_{n-1},
\end{equation}
with $H_{-1}=0$. We remark that it is the operator $A^t\bar A^t$ that intertwines $V\otimes d$ with $d\otimes V$, see~\cite[Lemma~2.3]{HaNe22}, and hence the restriction of $(A^t\bar A^t)^{\otimes n}$ to $H_n$ gives an isomorphism $H_{n}[1]\cong [1]H_{n}$.

It follows from \eqref{eq:fusion} that if we let $\F_{P,+}:=\bigoplus^\infty_{n=1}H_n\subset\F_P$, then we get a unitary
$$
\tilde W\colon (\F_{P,+}\otimes\F_P)\oplus(\F_P\otimes \F_P)\to (\F_P\otimes\F_{P,+})\oplus(\F_P\otimes \F_P)
$$
defined as the composition of the isomorphisms
\begin{multline*}
(\F_{P,+}\otimes\F_P)\oplus(\F_P[1]\otimes \F_P)\cong (\F_P\otimes H_1)\otimes\F_P\cong \F_P\otimes (H_1\otimes\F_P)\\ \cong (\F_P\otimes\F_{P,+})\oplus(\F_P\otimes [1]\F_P)=(\F_P\otimes\F_{P,+})\oplus(\F_P [1]\otimes\F_P).
\end{multline*}
We will rather think of it as a partial isometry on $(\F_{P}\otimes\F_P)\oplus(\F_P\otimes \F_P)$ such that
\begin{equation}\label{eq:iso-coiso}
1-\tilde W^*\tilde W=\begin{pmatrix}
    e_0\otimes 1 & 0\\ 0 & 0
\end{pmatrix},\qquad 1-\tilde W\tilde W^*=\begin{pmatrix}
    1\otimes e_0 & 0\\ 0 & 0
\end{pmatrix}.
\end{equation}

This partial isometry was introduced in~\cite{ArKa21} for a particular class of Temperley--Lieb polynomials.
To get an explicit formula for it we need to fix the unitary isomorphisms~\eqref{eq:fusion}.

Consider the embedding maps
$$
i_R\colon H_{n}\to H_{n-1}\otimes H_1\qquad\text{and}\qquad i_L\colon H_n\to H_1\otimes H_{n-1}.
$$
Define equivariant isometric maps
$$
V_R\colon H_n[1]\to H_{n+1}\otimes H_1\qquad\text{and}\qquad V_L\colon [1]H_n\to H_1\otimes H_{n+1}
$$
by the following formulas, cf.~\cite[Section 4]{HaNe22}:
$$
V_R:=([2]_q\phi(n+1))^{1/2}(f_{n+1}\otimes1)(1^{\otimes n}\otimes v),
$$
$$
V_L:=([2]_q\phi(n+1))^{1/2}(1\otimes f_{n+1})(v\otimes 1^{\otimes n}),
$$
where $v\colon \C\to H_1\otimes H_1$ is the isometry $1\mapsto  [2]_q^{-1/2}P$ and $\phi$ is given by~\eqref{eq:phi(n)}. Then
\begin{equation}\label{eq:tildeW}
\tilde W=\begin{pmatrix}
    (1\otimes i_L^*)(i_R\otimes1) & (1\otimes i_L^*)(V_R\otimes1) \\
    (1\otimes V_L^*)(i_R\otimes1) & (1\otimes V_L^*)(V_R\otimes1)
  \end{pmatrix}.    
\end{equation}

\begin{lemma}[cf.~{\cite[Lemma~6.4]{ArKa21}}]
If $P = \sum_{i=1}^{m} a_i X_iX_{m-i+1}$, $\vert a_i a_{m-i+1} \vert =1$, then \[ \tilde W = \sum_{i=1}^m \begin{pmatrix} R_i^* \otimes L_i &   a_i \phi^{1/2} R_i \otimes  L_{m-i+1} \\
 \overline{a}_i R_{i}^* \otimes L_{m-i+1}^*\phi^{1/2} & a_i \overline{a}_{m-i+1}\phi^{1/2} R_{i} \otimes L_{i}^* \phi^{1/2},
\end{pmatrix}.
\]
\end{lemma}

\bp
First, we claim that $V_L^*(\xi_i\otimes\eta)=\phi(n+1)^{1/2}\bar a_i S_{m-i+1}^*\eta$ for $\eta\in H_{n+1}$.
Indeed, if $\zeta\in [1]H_n$, then
\begin{align*}
 (V_L^*(\xi_i\otimes\eta),\zeta)&=\Big(\xi_i\otimes\eta,\phi(n+1)^{1/2}(1\otimes f_{n+1})\sum_ka_k\xi_k\otimes\xi_{m-k+1}\otimes\zeta\Big) \\
   & =\phi(n+1)^{1/2}(\eta, a_i S_{m-i+1}\zeta) =\phi(n+1)^{1/2}\bar a_i(S_{m-i+1}^*\eta, \zeta),
\end{align*}
proving the claim.

We are going to use also that
$$
i_R\xi=\sum_i R_i^*\xi\otimes\xi_i\quad\text{for}\quad \xi\in H_n,\qquad
i_L^*\xi=f_n\xi\quad\text{for}\quad \xi\in H_1\otimes H_{n-1}.
$$

For $\xi\in H_n$, $\eta\in H_k$, we compute:
\begin{align*}
(1\otimes i_L^*)(i_R\otimes1)(\xi\otimes\eta)&=\sum_iR_i^*\xi\otimes i_L^*(\xi_i\otimes\eta)=\sum_i(R_i^*\otimes S_i)(\xi\otimes\eta),\\
(1\otimes V_L^*)(i_R\otimes1)(\xi\otimes\eta)&=\sum_iR_i^*\xi\otimes V_L^*(\xi_i\otimes\eta)=\phi(k)^{1/2}\sum_i\bar a_i(R_i^*\otimes S^*_{m-i+1})(\xi\otimes\eta),\\
(1\otimes i_L^*)(V_R\otimes1)(\xi\otimes\eta)&=\phi(n+1)^{1/2}\sum_ia_if_{n+1}(\xi\otimes\xi_i)\otimes i_L^*(\xi_{m-i+1}\otimes\eta) \\
&=\phi(n+1)^{1/2}\sum_i a_i(R_i\otimes S_{m-i+1})(\xi\otimes\eta),\\
(1\otimes V_L^*)(V_R\otimes1)(\xi\otimes\eta)&=\phi(n+1)^{1/2}\sum_ia_if_{n+1}(\xi\otimes\xi_i)\otimes V_L^*(\xi_{m-i+1}\otimes\eta) \\
&=\phi(n+1)^{1/2}\phi(k)^{1/2}\sum_i a_i\bar a_{m-i+1}(R_i\otimes S_i^*)(\xi\otimes\eta).
\end{align*}
The lemma follows now from~\eqref{eq:tildeW}.
\ep

As $\phi(n)\to q$, we see in particular that the partial isometry
$$
W := \sum_{i=1}^m \begin{pmatrix} T_i^* \otimes S_i &   a_i \phi^{1/2} T_i \otimes  S_{m-i+1} \\
 \overline{a}_i T_{i}^* \otimes T_{m-i+1}^*\phi^{1/2} & a_i \overline{a}_{m-i+1}\phi^{1/2} T_{i} \otimes S_{i}^* \phi^{1/2}
 \end{pmatrix}\in \TT_{P^\dagger}\otimes\TT_P
$$
is a lifting of the element $w$ defined by~\eqref{eq:w}. Hence $w\in \OO_{P^\dagger}\otimes\OO_P$ is indeed unitary, which proves Lemma~\ref{lem:unitary}.

We note also that similar computations give a formula for $\tilde W$ for every Temperley--Lieb polynomial $P$. Since the expression is not particularly beautiful and not needed, we omit it. In fact, we will never use any explicit expression for $\tilde W$ for any $P$, the only property of $\tilde W$ in addition to~\eqref{eq:iso-coiso} that we need is that it defines an element of $\TT_{P^\dagger}\otimes\TT_P$. 

\smallskip

Consider now the unitary representation
$$
U_P:=\bigoplus_{n\ge0}U_n
$$
of $\GA$ on $\F_P$. It defines a right action $\TT_P\to\TT_P\otimes C(\GA)$ of $\GA$ on $\TT_P$ by $S\mapsto U_P(S\otimes 1)U_P^*$, which at the level of generators is given by $S_j\mapsto\sum_i S_i\otimes v_{ij}$. This action passes to an action $\delta_P\colon \OO_P\to\OO_P\otimes C(\GA)$ of $\GA$ on $\OO_P$.

Consider also the unitary representation
$$
U'_P:=(U_P\otimes U_P)\oplus (U_P\otimes d\otimes U_P)= (U_{P})_{13}({U_P})_{23}\oplus (U_{P})_{13}(1\otimes1\otimes d)({U_P})_{23}
$$
of $\GA$ on $(\F_P\otimes\F_P)^2$. This representation defines a right action of $\GA$ on $\K(\F_P^2)\otimes\TT_P$ by $S\mapsto U'_P(S\otimes 1){U'_P}^*$, which then passes to $\K(\F_P^2)\otimes\OO_P$. By construction the partial isometry~$\tilde W$ is a self-intertwiner of~$U'_P$.  It follows that the image of $\tilde W$ in $M(\K(\F_P^2)\otimes\OO_P)$ is invariant under the $\GA$-action. 

We thus see that the homomorphism~\eqref{eq:ext} is $\GA$-equivariant with respect to the actions 
$\iota\otimes\delta_P\colon C_0(\R)\otimes\OO_P\to C_0(\R)\otimes\OO_P\otimes C(\GA)$ and
$$
\K(\F_P^2)\otimes\OO_P\to M(\K(\F_P^2)\otimes\OO_P\otimes C(\GA)),\quad a\otimes s\mapsto ({U''_P})_{13}(a\otimes\delta_P(s))({U''_P})^*_{13},
$$
where ${U''_P}:=U_P\oplus (U_P\otimes d)$. It therefore defines a class in $KK_1^{\GA}(C_0(\R)\otimes\OO_P,\OO_P)$, see, e.g.,~\cite{NV} for a discussion of quantum group equivariant $KK$-theory. We are now ready to formulate the main technical result of the paper.

\begin{theorem}\label{thm:main-equi}
For every Temperley--Lieb polynomial $P$, the lift of the class $\beta^{-1}\otimes\delta\otimes_{\OO_{P^\dagger}}\Delta\in KK_1(C_0(\R)\otimes\OO_P,\OO_P)$ to an element of $KK_1^{\GA}(C_0(\R)\otimes\OO_P,\OO_P)$, as described above, equals $\beta^{-1}\otimes1$.
\end{theorem}

Before turning to the proof, let us briefly discuss the class $\beta^{-1}\otimes\delta\otimes_{\OO_P}\Delta \in KK_1(C_0(\R)\otimes\OO_{P^\dagger},\OO_{P^\dagger})$. Similarly to the discussion above, it lifts to a class in $KK_1^{\GA}(C_0(\R)\otimes\OO_{P^\dagger},\OO_{P^\dagger})$, the main difference being that now all actions of $\GA$ are on the left, implemented by exactly the same unitaries $U_P$ and $U_P'$ as before. For example, the left action of $\GA$ on $\TT_{P^\dagger}$ is defined by $T\mapsto (U_P)^*_{21}(1\otimes T)(U_P)_{21}$, where we view $\TT_{P^\dagger}$ as a C$^*$-subalgebra of $B(\F_P)$ represented by the right creation operators $R_i$. As should become clear later, the same arguments as in the proof of Theorem~\ref{thm:main-equi} show that the lifted class equals $-\beta^{-1}\otimes1\in KK_1^{\GA}(C_0(\R)\otimes\OO_{P^\dagger},\OO_{P^\dagger})$, where the minus sign appears because the image of $W$ in $\OO_{P^\dagger}\otimes\TT_P$ is an isometry, while its image in $\TT_{P^\dagger}\otimes\OO_P$ is a coisometry. Altogether, this implies that the elements $\delta$ and $\Delta$ indeed implement a $KK$-duality between $\OO_P$ and $\OO_{P^\dagger}$.

\subsection{Reduction to \texorpdfstring{$m=2$}{m=2}}

For a fixed Temperley--Lieb polynomial $P=\sum^m_{i,j=1}a_{ij}X_iX_j$, with $A\bar A$ unitary, consider the noncommutative polynomial $Q:=q^{-1/2}X_1X_2-q^{1/2}X_2X_1$, where $0<q\le1$ is such that $q+q^{-1}=\Tr(A^*A)$. The corresponding quantum group $\tilde O^+_Q$ is $U_q(2)$. By the main result of~\cite{Mro14}, see also~\cite{HaNe22}, the quantum groups $\GA$ and $U_q(2)$ are monoidally equivalent, that is, they have monoidally equivalent categories of unitary representations. We will not need an explicit description of such an equivalence, it is important only to realize that it can be chosen to transform the intertwiner $v\colon \C\to V\otimes V$, $1\mapsto [2]_q^{-1/2}P$, in $\Rep \GA$ into the similarly defined intertwiner in $\Rep U_q(2)$. This already implies that it maps the subproduct system $\HH_P$ into a subproduct system canonically isomorphic to $\HH_Q$, and under this isomorphism the partial isometry~\eqref{eq:tildeW} for $P$ transforms into the similarly defined partial isometry for $Q$.

Furthermore, the monoidal equivalence of the representation categories defines an equivalence of categories of $\GA$- and $U_q(2)$-C$^*$-algebras, and even more, an equivalence of the corresponding equivariant $KK$-categories $KK^{\GA}$ and $KK^{U_q(2)}$, see~\cite[Theorem~8.5]{Voi}. By \cite[Corollary~2.8]{HaNe21}, this equivalence maps $\OO_P$ into an object canonically isomorphic to $\OO_Q$. 

It follows that under the equivalence of the equivariant $KK$-categories the lift of the class $\beta^{-1}\otimes\delta\otimes_{\OO_{P^\dagger}}\Delta\in KK_1(C_0(\R)\otimes\OO_P,\OO_P)$ to an element of $KK_1^{\GA}(C_0(\R)\otimes\OO_P,\OO_P)$ described before Theorem~\ref{thm:main-equi} corresponds to the similarly defined element of $KK_1^{U_q(2)}(C_0(\R)\otimes\OO_Q,\OO_Q)$. Therefore it suffices to prove Theorem~\ref{thm:main-equi} for the polynomials $Q=q^{-1/2}X_1X_2-q^{1/2}X_2X_1$.

\subsection{Computations for \texorpdfstring{$U_q(2)$}{Uq(2)}}

For the rest of this section we consider the polynomial $P:=q^{-1/2}X_1X_2-q^{1/2}X_2X_1$ for a fixed $0<q\le1$. Consider the elements
$$
\alpha:=v_{22}^*,\qquad \gamma:=v_{21}
$$
in $C(\GA)=C(U_q(2))$. Then
$$
V=\begin{pmatrix}
    d & 0 \\ 0 & 1
\end{pmatrix}\begin{pmatrix}
    \alpha & - q\gamma^*\\ \gamma & \alpha^*
\end{pmatrix},
$$
cf.~\cite[Appendix A]{HaNe22}. The elements $\alpha$ and $\gamma$ generate a copy of $C(SU_q(2))$ in $C(U_q(2))$, while the element $d$ is central, so as a C$^*$-algebra $C(U_q(2))$ decomposes as $C^*(d)\otimes C(SU_q(2))\cong C(\T)\otimes C(SU_q(2))$.

Consider the one-dimensional torus
$$
\T:=\begin{pmatrix} \mathbb{T} & 0 \\ 0&1
\end{pmatrix} \subset U_q(2).
$$
In other words, the subgroup $\T\subset U_q(2)$ is defined by the homomorphism $C(U_q(2))\to C(\T)$, $v_{12}\mapsto0$, $v_{21}\mapsto0$, $v_{11}\mapsto z$, $v_{22}\mapsto 1$, that is, by
$$
d\mapsto z,\quad \alpha\mapsto 1,\quad \gamma\mapsto 0.
$$
This subgroup acts on $C(U_q(2))$ by left translations and the fixed point algebra is
$$
C({\T}\backslash U_q(2))=C(SU_q(2)).
$$
From Theorem~\ref{thm: univ_relations} we then get a right $U_q(2)$-equivariant isomorphism
\begin{equation}\label{eq:Op-iso}
C({\T}\backslash U_q(2))\cong\OO_P,\quad \alpha\mapsto s_2^*,\quad \gamma\mapsto s_1,
\end{equation}
which is a particular case of~~\cite[Theorem~2.8]{HaNe22}.

On the other hand, we also have the action of $\T$ on $C(U_q(2))$ by right translations. Its restriction to $C(SU_q(2))=C({\T}\backslash U_q(2))$ is given by
\begin{equation}\label{eq:torus-action}
C(SU_q(2))\to C(SU_q(2))\otimes C(\T),\quad \alpha\mapsto\alpha\otimes1,\quad\gamma\mapsto\gamma\otimes z.
\end{equation}
We want to stress that this is \emph{not} the action by right translations of the standard maximal torus $T\subset SU_q(2)$, but the action of $T/\{\pm1\}$ by conjugation.

\smallskip

Consider now the counit $\eps_q$  for the quantum group $SU_q(2)$, so
$$
\eps_q\colon  C(SU_q(2))\to \mathbb{C},\quad \eps_q(\alpha)=1,\quad \eps_q(\gamma)=0.
$$ 
It is $\mathbb{T}$-equivariant, hence it gives a class $[ \eps_q ] \in KK^{\mathbb{T}}(C(SU_q(2)),\mathbb{C})$. We also consider the $*$-homomorphism
\[ \pi_q\colon  C_0(\mathbb{R}) \to C_0(\mathbb{R}) \otimes C(SU_q(2)), \quad x \mapsto x \otimes 1.\] 
Identify the ring of Laurent polynomials $\mathbb Z [t,t^{-1}]$ with the representation ring $R(\mathbb T)$. We use the maps $\eps_q$ and $\pi_q$ to prove the following auxiliary result.

\begin{lemma}
\label{lem:Tequiv_Khomology}
The maps
\begin{align*}
- \otimes [\eps_q ] &: KK_1^{\mathbb{T}} (C_0(\mathbb{R}),\mathbb{C}) \to KK_1^{\mathbb{T}} (C_0(\mathbb{R}) \otimes C(SU_q(2)),\mathbb{C}),\\
\left[\pi_{q} \right]\otimes_{C_0(\R)\otimes C(SU_q(2))} - &: KK_1^{\mathbb{T}} (C_0(\mathbb{R}) \otimes C(SU_q(2)),\mathbb{C}) \to KK_1^{\mathbb{T}} (C_0(\mathbb{R}),\mathbb{C})
\end{align*}
are inverse to each other, where $\T$ acts trivially on $C_0(\R)$ and it acts by~\eqref{eq:torus-action} on $C(SU_q(2))$. Therefore, we obtain an $R(\mathbb{T})$-module isomorphism of $K^0_{\mathbb{T}}(C(SU_q(2)))$ with~$R(\mathbb{T})$.
\end{lemma}

For $q=1$ the last statement of the lemma can be deduced from the six-term exact sequence in $\mathbb{T}$-equivariant $K$-homology for the compact pair $(S^3, S^1)$, see \cite{Se68,Ka88}. The proof below can be seen as an adaptation of this to the noncommutative setting.

\begin{proof}[Proof of Lemma~\ref{lem:Tequiv_Khomology}]
It is straightforward to see that $\left[\pi_{q} \right]$ is a left inverse to $ [\eps_q ] $ for the Kasparov product over $C_0(\mathbb{R})\otimes C(SU_q(2))$. Therefore the $R(\T)$-module 
$$
K^0_{\mathbb{T}}(C(SU_q(2)))=KK^{\mathbb{T}}(C(SU_q(2)),\C)\cong KK_1^{\mathbb{T}} (C_0(\mathbb{R}) \otimes C(SU_q(2)),\mathbb{C})
$$ 
contains the representation ring $R(\mathbb{T})$ as a direct summand, since by Bott periodicity we have $KK_1^\T(C_0(\R),\C)\cong K_0^\T(\C)\cong R(\T)$. In order to prove the lemma it suffices to show that $K^0_{\mathbb{T}}(C(SU_q(2)))$ is isomorphic to $R(\T)$ as an $R(\T)$-module. Indeed, if $M\oplus R(\T)\cong R(\T)$ for an $R(\T)$-module $M$, then $M$ must be zero, since $R(\T)$ has no divisors of zero and therefore any nonzero $R(\T)$-module morphism $R(\T)\to R(\T)$ is injective.

To compute $K^0_\T(C(SU_q(2)))$ we will use the composition series of $C(SU_q(2))$ which was discussed already in~\cite[Appendix~2]{Wo87}. In order to deal simultaneously with the cases $q=1$ and $q<1$ it is convenient to introduce the unital C$^*$-algebra $C(\bar\DD_q)$ generated by an element $Z_q$ such that $\|Z_q\|\le1$ and $1-Z^*_qZ_q=q^2(1-Z_qZ^*_q)$. For $q=1$ this is the algebra of continuous functions on the closed unit disc $\bar\DD\subset\C$. Denote by $C_0(\DD_q)$ the kernel of the homomorphism $C(\bar\DD_q)\to C(\T)$, $Z_q\mapsto z$.

We have a short exact sequence
\begin{equation}\label{eq:SU2-series}
\xymatrix{0\ar[r] & C(\mathbb{T}) \otimes C_0(\DD_q) \ar[r]& C(SU_q(2)) \ar[r]& C(\mathbb{T}) \ar[r]& 0},
\end{equation}
see, e.g., \cite[Section~3]{NTpoisson}. Here the homomorphism $C(SU_q(2)) \to C(\mathbb{T})$ is given by $\alpha\mapsto z$ and $\gamma\mapsto0$, while $C(\mathbb{T})\otimes C_0(\DD_q)$ is considered as a subalgebra of $C(SU_q(2))$ using the embedding $C(SU_q(2))\hookrightarrow C(\T)\otimes C(\bar\DD_q)$ defined by $\alpha\mapsto z\otimes Z^*_q$, $\gamma\mapsto -z\otimes(1-Z_qZ_q^*)^{1/2}$.

The short exact sequence~\eqref{eq:SU2-series} is $\T$-equivariant with respect to the trivial action on the last term $C(\T)$ and the diagonal action on $C(\mathbb{T})\otimes C_0(\DD_q)$, with $\T$ acting on the first tensor factor~$C(\T)$ by right translations, that is,  $C(\T)\to C(\T)\otimes C(\T)$, $z\mapsto z\otimes z$, and acting on the second factor~$C_0(\DD_q)$ by the restriction of the action $C(\bar\DD_q)\to C(\bar\DD_q)\otimes C(\T)$, $Z_q\mapsto Z_q\otimes z$.

The $\T$-C$^*$-algebra $C_0(\DD_q)$ is $KK^\T$-equivalent to $\C$. For $q<1$ this is true simply because $C_0(\DD_q)\cong\K$, see again \cite[Section~3]{NTpoisson}, while for $q=1$ this follows from Kasparov's $KK$-theoretic version of Bott periodicity~\cite[Theorem~7]{Ka80}. We therefore get from~\eqref{eq:SU2-series} a six-term exact sequence 
\[ \xymatrix{ & K^0_{\mathbb{T}}(C(\mathbb{T}))  \ar[d]&  K^0_{\mathbb{T}}(C(SU_q(2)))  \ar[l] & K^0_{\mathbb{T}} (C(\mathbb{T})_{\mathrm{triv}}) \ar[l] \\
& K^1_{\mathbb{T}} (C(\mathbb{T})_{\mathrm{triv}}) \ar[r] &  K^1_{\mathbb{T}}(C(SU_q(2))) \ar[r] & K^1_{\mathbb{T}} (C(\mathbb{T})) \ar[u]}, \]
where we have used the subscript \enquote{$\mathrm{triv}$} to distinguish the case when the action on $C(\mathbb{T})$ is the trivial one from the case of the action by right translations.

We compute, using first $KK$-duality for the circle and then the Green--Julg theorem for the $K$-theory of crossed products \cite{Ju81}:
\begin{multline*}
K^i_{\mathbb{T}} (C(\mathbb{T})_{\mathrm{triv}}) =  KK_i^{\mathbb{T}} (C(\mathbb{T})_{\mathrm{triv}},\mathbb{C})) \cong KK_{i-1}^{\mathbb{T}} (\mathbb{C}, C(\mathbb{T})_{\mathrm{triv}}) \\ \cong K_{i-1} (C(\mathbb T) \otimes c_0(\Z)) \cong \bigoplus_{\mathbb{Z}} \mathbb{Z}  \cong R(\mathbb{T}).
\end{multline*}
It is not difficult to see that this is an isomorphism of $R(\T)$-modules.

Similarly, but this time for the nontrivial action of $\T$ on itself by translations and $\mathbb T$-equivariant $KK$-duality for the circle, we compute:
\begin{multline*} 
K^i_{\mathbb{T}} (C(\mathbb{T})) =  KK_i^{\mathbb{T}} (C(\mathbb{T}),\mathbb{C})) \cong KK_{i-1}^{\mathbb{T}} (\mathbb{C}, C(\mathbb{T}))\\ \cong K_{i-1} (C(\mathbb{T})\rtimes \mathbb{T}) \cong K_{i-1}(\mathcal{K}) \cong \begin{cases} \mathbb{Z}, & i=1, \\
0, & i=0. \end{cases}
\end{multline*}
Again, it is not difficult to see that the $R(\T)$-module structure on $\Z$ is given by the augmentation homomorphism $R(\T)=\Z[t,t^{-1}]\to\Z$, $t^{\pm1}\mapsto1$.

Hence, the above six-term exact sequence reduces to 
\[ \xymatrix{ & 0  \ar[d]&  K^0_{\mathbb{T}}(C(SU_q(2)))  \ar[l] & R (\mathbb{T}) \ar[l] \\
& R(\mathbb{T}) \ar[r] &  K^1_{\mathbb{T}}(C(SU_q(2))) \ar[r] & \mathbb{Z} \ar[u]} .\]
Since there are no nonzero $R(\mathbb T)$-module morphisms $\mathbb{Z} \to R(\mathbb T)$, we obtain the required isomorphism $K^0_{\mathbb{T}}(SU_q(2)) \cong R(\mathbb{T}).$ 
\end{proof}

Turning to the proof of Theorem~\ref{thm:main-equi} for $P=q^{-1/2}X_1X_2-q^{1/2}X_2X_1$, we identify $\OO_P$ with $C(SU_q(2))$ by~\eqref{eq:Op-iso} and denote by $W_P$ the image of~$\tilde W$ in $M(\K(\F_P^2)\otimes C(SU_q(2)))$. We need to show that the class of the extension in $KK_1^{U_q(2)}(C_0(\R)\otimes C(SU_q(2)),C(SU_q(2)))$ defined by the homomorphism
\begin{equation*}
\label{eq:hom_Calkin}
C_0(\mathbb{R})\otimes C(SU_q(2)) \to \mathcal{Q}((\F_P\oplus\F_P[1])\otimes C(SU_q(2))),
\end{equation*}
$$
z\otimes 1  \mapsto W_P^*, \quad 1 \otimes \alpha \mapsto \begin{pmatrix}
L_2^* \otimes 1 & 0\\
0 & L_2^*\otimes 1
\end{pmatrix}, \quad 1 \otimes \gamma \mapsto\begin{pmatrix}
L_1\otimes 1 & 0\\
0 & L_1\otimes 1
\end{pmatrix},
$$
equals $\beta^{-1}\otimes1$. We remind that the notation $\F_P[1]$ is used to indicate that we consider the representation $U_P\otimes d$ of $U_q(2)$ on the second copy of $\F_P$.

To simplify this class, we use that $C(SU_q(2))=C(\T\backslash U_q(2))$
is isomorphic to the induced algebra $\mathrm{Ind}_{\mathbb{T}}^{U_q(2)}(\mathbb{C})$ and apply Frobenius reciprocity
$KK^{U_q(2)}(B,C(SU_q(2)))\cong KK^\T(B,\C)$, see~\cite[Proposition~4.7]{NV}. Explicitly, this isomorphism is the composition of the forgetful homomorphism $KK^{U_q(2)}(B,C(SU_q(2)))\to KK^\T(B,C(SU_q(2)))$ with the map induced by the counit $\eps_q\colon C(SU_q(2))\to\C$. It follows that we need to show that the class of the extension in $KK_1^\T(C_0(\R)\otimes C(SU_q(2)),\C)$ defined by the homomorphism
$$
C_0(\mathbb{R})\otimes C(SU_q(2))  \to \mathcal{Q}(\F_P\oplus\F_P[1]), 
$$
$$
\qquad
z\otimes 1 \mapsto (\iota\otimes \eps_q)(W_P)^*,\quad
1 \otimes \alpha \mapsto \begin{pmatrix}
L_2^* & 0\\
0 & L_2^*
\end{pmatrix}, \quad 1 \otimes \gamma \mapsto\begin{pmatrix}
L_1 & 0\\
0 & L_1
\end{pmatrix},
$$
equals $\beta^{-1}\otimes[\eps_q]$.

By applying the map $[\pi_{q}]\otimes_{C_0(\R)\otimes C(SU_q(2))} - $ and using Lemma~\ref{lem:Tequiv_Khomology} we see that this is equivalent to showing that the class of the extension in $KK_1^\T(C_0(\R),\C)$ defined by the homomorphism 
\begin{equation}\label{eq:ext2}
C_0(\mathbb{R}) \to \mathcal{Q}(\F_P\oplus\F_P[1]), \quad z \mapsto (\iota\otimes \eps_q)(W_P)^*,
\end{equation}
equals $\beta^{-1}$, where we now view the inverse of the Bott class as an element of $KK_1^\T(C_0(\R),\C)$.

As was already used in the proof of Lemma~\ref{lem:Tequiv_Khomology}, we have $KK_1^\T(C_0(\R),\C)\cong R(\T)$, and explicitly the isomorphism is given by the equivariant index. Let us set $V_P:= (\iota\otimes \eps_q)(W_P)$. From~\eqref{eq:iso-coiso} we know that $V_PV^*_P =1 $ and $1-V^*_PV_P = \begin{pmatrix}
    e_0 & 0 \\ 0 & 0
\end{pmatrix}$, where $e_0$ is the projection onto~$H_0$. As the representation of  $\T=\begin{pmatrix} \mathbb{T} & 0 \\ 0&1
\end{pmatrix} \subset U_q(2)$ on $H_0\oplus0\subset\F_P\oplus\F_P[1]$ is trivial and $\dim H_0=1$, we see that $V^*_P$ has $\T$-equivariant index $-1\in R(\T)$, hence the class of~\eqref{eq:ext2} equals~$\beta^{-1}$ in $KK_1^\T(C_0(\R),\C)$. This completes the proof of Theorem~\ref{thm:main-equi}, hence also the proof of the $KK$-duality for all Temperley--Lieb polynomials.

\bigskip

\section{KMS-state}\label{sec:KMS}

The Cuntz--Pimsner algebra $\OO_P$ has a canonical KMS-state. In order to describe it we need a couple of results on invariant states. But first we need to recall a few notions.

Given a right action $\alpha\colon B\to B\otimes C(G)$ of a compact quantum group $G$ on a C$^*$-algebra $B$, we can define, for every $\phi\in C(G)^*$, an operator
$$
\phi* - \colon B\to B,\quad \phi*b:=(\iota\otimes\phi)\alpha(b).
$$
A similar definition makes of course also sense for left actions. We will need this construction for the characters $\phi=\rho^{it}$, where $\rho$ is the Woronowicz character of $G$. In general these characters are defined only on the universal completion $C_u(G)$ of $\C[G]$ rather than on $C(G)$ and, correspondingly, the automorphisms $\rho^{it}* -$ are a priori defined only on a universal form $B_u$ of the $G$-C$^*$-algebra $B$, see~\cite[Section~4]{DC}. However, for a class of actions they descend to $B$.

Namely, recall that $\alpha$ is called reduced, if $\ker\alpha=0$. We recall that $C(G)$ denotes the reduced function algebra of $G$, hence the actions of $G$ on $C(G)$ by left and right translations are reduced. For the reduced actions the kernel of the canonical map $B_u\to B$ equals
$
\{b\in B_u\mid (\iota\otimes h)\alpha_u(b^*b)=0\},
$
where $h\in C_u(G)^*$ is the Haar state and $\alpha_u\colon B_u\to B_u\otimes C_u(G)$ is the universal form of the action~$\alpha$. From this one can see that the kernel of $B_u\to B$ is invariant under $\rho^{it}* -$, so we still get well-defined automorphisms of $B$. 

Next, recall that given two compact quantum groups $G_1$ and $G_2$, by a C$^*$-$G_1$-$G_2$-Galois object one means a unital C$^*$-algebra $B$ equipped with two commuting actions $\alpha_1\colon B\to C(G_1)\otimes B$ and $\alpha_2\colon B\to B\otimes C(G_2)$ that are both ergodic and free. Ergodicity means that the fixed point algebras $^{G_1}\!\!B:=\{b\mid \alpha_1(b)=1\otimes b\}$ and $B^{G_2}:=\{b\mid\alpha_2(b)=b\otimes1\}$ are trivial. The precise meaning of freeness will not be important for us, see, e.g., \cite[Section~1]{HaNe22} for details.

\begin{lemma}\label{lem:modular}
Assume $G_1$ and $G_2$ are compact quantum groups and $B$ is a C$^*$-$G_1$-$G_2$-Galois object in reduced form. Let $\omega$ be the unique invariant state on $B$. Then $\omega$ is a $\sigma$-KMS$_{-1}$-state, where $\sigma$ is given~by
$$
\sigma_t(b)=\rho_2^{it}*b*\rho_1^{it}
$$
and $\rho_1$ and $\rho_2$ are the Woronowicz characters of $G_1$ and $G_2$.
\end{lemma}

\bp
This follows from the orthogonality relations for $\omega$, see \cite[Theorem~2.3.11]{NT}.
\ep

\begin{lemma}\label{lem:state}
Assume $G$ is a compact quantum group, $\alpha\colon B\to B\otimes C(G)$ is a right action of $G$ on a unital C$^*$-algebra $B$, $U\in B(H)\otimes C(G)$ is an irreducible unitary representation of $G$ and $S\colon H\to B$ is a $G$-equivariant linear map, so that
$$
(S\otimes\iota)U(\xi\otimes 1)=\alpha(S(\xi))\quad\text{for all}\quad \xi\in H.
$$
Assume also that if $(\zeta_i)_i$ is an orthonormal basis in $H$, then $\sum_i S(\zeta_i)S(\zeta_i)^*=1$. Then, for any $G$-invariant state $\omega$ on $B$ and all $\xi,\zeta\in H$, we have
$$
\omega(S(\xi) S(\zeta)^*)=\frac{(\rho_U\xi,\zeta)}{\dim_q U},
$$
where $\rho_U:=(\iota\otimes\rho)(U)\in B(H)$ is the operator defined by the Woronowicz character of $G$ and $\dim_qU:=\Tr(\rho_U)$ is the quantum dimension of $U$.
\end{lemma}

\bp
Consider the right action $T\mapsto U(T\otimes1)U^*$ of $G$ on $B(H)$. We then have a $G$-equivariant ucp map $\theta\colon B(H)\to B$, $\xi\otimes\bar\zeta\mapsto S(\xi) S(\zeta)^*$, where we identified $B(H)$ with $H\otimes\bar H$. Hence $\omega\circ \theta$ is a $G$-invariant state on $B(H)$. Since $U$ is irreducible, such an invariant state is unique. It is well-known that this state is $(\dim_q U)^{-1}\Tr(- \rho_U)$, see, e.g., the proof of \cite[Lemma~1.4.8]{NT}. This gives the result.
\ep

We are now ready to describe the canonical KMS-state on $\OO_P$.

\begin{prop}
\label{prop:KMS}
Consider a Temperley--Lieb polynomial $P=\sum^m_{i=1}a_iX_iX_{m-i+1}$ satisfying $|a_ia_{m-i+1}|=1$. Then there is a unique $\GA$-invariant state $\omega$ on $\OO_P$. Explicitly, it is given by
\begin{align}
\omega(s_{j_n}^*\dots s_{j_1}^*s_{i_1}\dots s_{i_m})&=\delta_{m,n}\frac{q^{-n}}{[n+1]_q}(f_n(\xi_{i_1}\otimes\dots\otimes \xi_{i_n}),\xi_{j_1}\otimes\dots\otimes \xi_{j_n}),\label{eq:KMS1}\\
\omega(s_{i_1}\dots s_{i_m}s_{j_n}^*\dots s_{j_1}^*)
&=\delta_{m,n}\frac{|a_{j_1}\dots a_{j_n}|^2}{[n+1]_q}(f_n(\xi_{i_1}\otimes\dots\otimes \xi_{i_n}),\xi_{j_1}\otimes\dots\otimes \xi_{j_n}).\label{eq:KMS2}
\end{align}
Furthermore, this state is KMS$_{-1}$ with respect to the dynamics $\sigma$ given by
$$
\sigma_t(s_j)=q^{it}|a_j|^{2it}s_j.
$$
\end{prop}

As before, the number $0<q\le1$ here is given  by~\eqref{eq:q} and $(\xi_i)_i$ is our fixed orthonormal basis in $H_1=\C^m$.

\bp
Similarly to \cite[Example 1.4.2]{NT}, the Woronowicz character of $\GA$ in the representation~$V$ of~$\GA$ on $H=\C^m$ is given by
$$
\rho_V=\begin{pmatrix}
|a_1|^2 & & 0\\
 & \ddots & \\
 0 & & |a_m|^2
\end{pmatrix}.
$$
The Cuntz--Pimsner algebra $\OO_P$ is a universal unital C$^*$-algebra with generators $s_1,\dots, s_m$ such that
$$
U=\begin{pmatrix}
q^{1/2}\bar{a}_1s_{m}^* & q^{1/2}\bar{a}_2s_{m-1}^* & \cdots & q^{1/2}\bar{a}_m s_{1}^*\\
s_{1} & s_{2} & \cdots & s_{m}
\end{pmatrix}
$$
is unitary. By \cite[Theorem~2.8]{HaNe22}, it is an $\GA$-invariant subalgebra of a C$^*$-$U_q(2)$-$\GA$-Galois object, and this object is automatically reduced by coamenability of $U_q(2)$. In particular, the action of $\GA$ on $\OO_P$ is ergodic and therefore there is exactly one $\GA$-invariant state $\omega$. Using Lemma~\ref{lem:modular} we can then conclude that $\omega$ is a KMS$_{-1}$-state with respect to the dynamics $\sigma$ such that 
$$
\sigma_t(U)=\begin{pmatrix}
q^{-it} & 0\\
0 & q^{it}\end{pmatrix}U\rho_V^{it},
$$
that is, $\sigma_{t}(s_j)=q^{it}|a_j|^{2it}s_j$.

In order to get an explicit formula for $\omega$ we apply Lemma~\ref{lem:state}. Consider the $G$-equivariant map
$$
\tilde s\colon H^{\otimes n}\to\OO_P, \quad \tilde s(\xi_{i_1}\otimes\dots\otimes \xi_{i_n})=s_{i_1}\dots s_{i_n}.
$$
This map factors through $H_n$, that is, we have a map $s\colon H_n\to \OO_P$ such that $\tilde s=s(f_n-)$. If $(\zeta_i)_i$ is an orthonormal basis in $H_n$, then the element $\sum_i s(\zeta_i)s(\zeta_i)^*\in\OO_P$ is $\GA$-invariant, hence it is scalar. By Lemma~\ref{lem:state} it follows that
\begin{align*}
\omega(s_{i_1}\dots s_{i_n}s_{j_n}^*\dots s_{j_1}^*)
&=\lambda_n(\rho_V^{\otimes n}f_n(\xi_{i_1}\otimes\dots\otimes \xi_{i_n}),f_n(\xi_{j_1}\otimes\dots\otimes \xi_{j_n}))\\
&=\lambda_n|a_{j_1}\dots a_{j_n}|^2(f_n(\xi_{i_1}\otimes\dots\otimes \xi_{i_n}),\xi_{j_1}\otimes\dots\otimes \xi_{j_n})
\end{align*}
for some number $\lambda_n>0$ that depends only on $n$. Summing up over all $i_1=j_1,\dots,i_n=j_n$ we get
$$
1=\lambda_n\Tr(\rho_V^{\otimes n}f_n)=\lambda_n\dim_q U_n=\lambda_n[n+1]_q.
$$
This proves formula~\eqref{eq:KMS2} for $m=n$. That we get zero for $m\ne n$ follows from gauge-invariance of $\omega$, which in turn follows from $\GA$-invariance. Formula~\eqref{eq:KMS1} follows then from the KMS-condition.
\ep

\begin{remark}
If we view $\omega$ as a state on the Toeplitz algebra, the formula for $\omega$ can be written in the following way similar to the usual Cuntz--Krieger--Pimsner algebras. Consider the gauge-invariant subalgebra $\TT_P^\T$ of the Toeplitz algebra. Each space $H_n$ is invariant under $\TT_P^\T$ . Let $\pi_n$ be the representation of $\TT_P^\T$ on $H_n$. On $B(H_n)$ we have a unique right $\GA$-invariant state
$$
\psi_n:=\frac{\Tr(\rho_{U_n}-)}{[n+1]_q},\quad\text{with}\quad\rho_{U_n}=\rho_V^{\otimes n}|_{H_n}.
$$
Then
$$
\omega(S_{i_1}\dots S_{i_n}S_{j_n}^*\dots S_{j_1}^*)=\psi_k(\pi_k(S_{i_1}\dots S_{i_n}S_{j_n}^*\dots S_{j_1}^*))
$$
for all $k\ge n$. Therefore for any element $a$ of $\TT_P^\T$ that is a polynomial in the generators and their adjoints we have $\omega(a)=\psi_k(\pi_k(a))$ for all $k$ large enough.
\end{remark}

\bigskip

\section{A Fredholm module representative of the fundamental class}\label{sec:Fredholm}
In this section we adapt an idea from~\cite{GoMe} to construct a Fredholm module representative of the fundamental class of the $KK$-duality. Consequently, we can describe the generators of the $K$-homology of the Cuntz--Pimsner algebras by Fredholm modules.

\smallskip

Consider a Temperley--Lieb polynomial $P=\sum^m_{i=1}a_iX_iX_{m-i+1},$ $|a_ia_{m-i+1}|=1$. Throughout this section we assume that $\sum_{i=1}^{m}|a_i|^2>2$, that is, the number $0<q\le1$ such that $\sum_{i=1}^{m}|a_i|^2=q+q^{-1}$ satisfies $q<1$. This excludes only the polynomials with $m=2$ and $|a_1|=|a_2|=1$.
We need this for the following result, which is a simple consequence of \cite[Lemma~4.2]{HaNe21}, see also \cite[Lemma~A.4]{VaVe07}.

\begin{lemma}
\label{lem:Fredholm}
There is a constant $C_1>0$, depending only on $q$, such that, for all $0\leq k<n$, one has $$\|(f_{k+1}\otimes 1_{n-k})(1\otimes f_n)-f_{n+1}\|\leq C_1q^k.$$
\end{lemma}

The construction of the Hilbert space underlying the Fredholm module representative relies on the KMS-state $\omega$ on $\mathcal{O}_P$, whose existence is ensured by Proposition~\ref{prop:KMS}. We denote its counterpart for the dual algebra $\mathcal{O}_{P^\dagger}$ by $\omega^{\dagger}$. We remark that $\OO_{P^\dagger}$ carries a left action of $\GA$ and a right action $\tilde O^+_{P^\dagger}$, but it is not difficult to check that the invariant states for these actions are the same, so there is no ambiguity in the definition of $\omega^\dagger$.

We write $L^2(\mathcal{O}_{P}, \omega)$ and $L^2(\mathcal{O}_{P^\dagger},\omega^{\dagger})$ for corresponding GNS-spaces. Consider the GNS-maps \[\Lambda\colon \mathcal{O}_P\to L^2(\mathcal{O}_P, \omega),\qquad \Lambda^{\dagger}\colon \mathcal{O}_{P^{\dagger}}\to L^2(\mathcal{O}_{P^\dagger}, \omega^{\dagger}),\] and the corresponding GNS-representations
\[\rho\colon \mathcal{O}_P\to B(L^2(\mathcal{O}_P,\omega)), \qquad \rho^{\dagger}\colon \mathcal{O}_{P^{\dagger}}\to B(L^2(\mathcal{O}_{P^\dagger},\omega^{\dagger})).\] 

For $n\geq 0$ and $\xi \in H_n$, we write $\xi^{\dagger}$ for the vector with the tensor factors of $\xi$ written in the opposite order. Then $\xi^{\dagger}$ lies in the $n$-th fiber of the subproduct system defined by $P^{\dagger}$. Further, for $0\leq k\leq n$ we can write $$\xi= \sum_i\xi [k]_i\otimes \xi [n-k]_i\in H_k\otimes H_{n-k}.$$ Similarly to Sweedler's sumless notation (see \cite[p. 52]{NT}), we omit the sum and simply write $\xi = \xi [k]\otimes \xi [n-k].$ 

For $\xi=\sum_\alpha c_\alpha\xi_{\alpha_1}\otimes\dots\otimes\xi_{\alpha_n}\in H_n$, put
$$
s_\xi:=\sum_\alpha c_\alpha s_{\alpha_1}\dots s_{\alpha_n}\in\OO_P\quad\text{and}\quad
t_{\xi^\dagger}:=\sum _\alpha c_\alpha t_{\alpha_n}\dots t_{\alpha_1}\in\OO_{P^\dagger}.
$$

The following result is a straightforward consequence of~\eqref{eq:KMS1}.

\begin{lemma}
The operator $V\colon\Fock \to L^2(\mathcal{O}_P,\omega)\otimes L^2(\mathcal{O}_{P^\dagger},\omega^{\dagger})$ given by \[ V\xi: = \frac{q^{n/2}}{(n+1)^{1/2}}\sum_{k=0}^{n}[k+1]_q^{1/2}[n-k+1]_q^{1/2}\Lambda(s_{\xi [k]})\otimes \Lambda^{\dagger}(t_{\xi [n-k]^{\dagger}}),\]
for $n\geq 0$ and $\xi \in H_n$, is an isometry.
\end{lemma}

We are now ready to describe a Fredholm module representative for the fundamental class $\Delta$ given in Definition~\ref{def:fundCl}.

\begin{theorem}\label{thm:Fredholm}
Consider a Temperley--Lieb polynomial $P=\sum^m_{i=1}a_iX_iX_{m-i+1},$ $|a_ia_{m-i+1}|=1$, with $\sum_{i=1}^{m}|a_i|^2>2$. Then the triple $\left(\mathcal{H}, \pi, F\right)$, where $\mathcal{H}:=L^2(\mathcal{O}_P,\omega)\otimes L^2(\mathcal{O}_{P^\dagger},\omega^{\dagger}),$ $\pi:=\rho \otimes \rho^{\dagger}$ and $F:=2VV^*-1$, defines an odd Fredholm module that represents $\Delta \in KK_1(\mathcal{O}_{P}\otimes \mathcal{O}_{P^\dagger},\mathbb C)$.
\end{theorem}

\begin{proof}
We show that for the extension $\tau\colon \mathcal{O}_P\otimes \mathcal{O}_{P^\dagger}\to \mathcal{Q}(\Fock)$ in \eqref{eq:fundext}, for every $x\in \mathcal{O}_P\otimes \mathcal{O}_{P^\dagger}$, it holds that 
\begin{equation}\label{eq:Fredholm2}
\tau(x)=V^*\pi(x)V+\mathcal{K}(\Fock).    
\end{equation}
Since $\tau$ is a $*$-homomorphism, it follows then that $[\pi(x),2VV^*-1]\in \mathcal{K}(\HH)$, for every $x\in \mathcal{O}_P\otimes \mathcal{O}_{P^\dagger}$. Thus, $(\mathcal{H},\pi,F)$ is an odd Fredholm module that represents $\Delta$, as it satisfies~\eqref{eq:Fredholm2}.

Now, since each $H_n$ is finite-dimensional, in order to prove \eqref{eq:Fredholm2}, it is enough to show that, for every $1\leq i\leq m$,
\begin{equation}\label{eq:Fredholm3}
\lim_{n\to \infty} \|(VL_i-\pi(s_i\otimes 1_{\mathcal{O}_{P^\dagger}})V)|_{H_n}\|=0,\end{equation}
and that similarly \[\lim_{n\to \infty} \|(VR_i-\pi(1_{\mathcal{O}_{P}}\otimes t_i)V)|_{H_n}\| =0. \]
We shall discuss the former case only, as the latter follows \emph{mutatis mutandis}. 

For $\xi\in H_n$, we have  
\[
VL_i\xi =\frac{q^{(n+1)/2}}{(n+2)^{1/2}}\sum_{k=0}^{n+1}[k+1]_q^{1/2}[n-k+2]_q^{1/2}\Lambda(s_{f_{n+1}(\xi_i\otimes \xi)[k]})\otimes \Lambda^{\dagger}(t_{f_{n+1}(\xi_i\otimes \xi)[n-k]^{\dagger}}),\]
and
\[
\pi(s_i\otimes 1_{\mathcal{O}_{P^\dagger}})V\xi =\frac{q^{n/2}}{(n+1)^{1/2}}\sum_{k=0}^{n}[k+1]_q^{1/2}[n-k+1]_q^{1/2}\Lambda(s_{f_{k+1}(\xi_i\otimes \xi[k])})\otimes \Lambda^{\dagger}(t_{\xi[n-k]^{\dagger}}).
\]
Therefore, 
\begin{equation}\label{eq:commutatorVL}
\begin{aligned}
& \|VL_i\xi-\pi(s_i\otimes 1_{\mathcal{O}_{P^\dagger}})V\xi\|^2= \frac{\|f_{n+1}(\xi_i\otimes \xi)\|^2}{n+2}\\
& \quad  + \sum_{k=1}^{n+1} \left \|\frac{1}{(n+2)^{1/2}}f_{n+1}(\xi_i\otimes \xi)-\frac{q^{-1/2}[k]_q^{1/2}}{(n+1)^{1/2}[k+1]_q^{1/2}}(f_{k+1}\otimes 1_{n-k})(\xi_i\otimes \xi) \right \|^2.
 \end{aligned}
\end{equation}

Now, for any integer $n \geq 0$, we have that 
\begin{equation}\label{eq:invRoots}
0 < \frac{1}{(n+1)^{1/2}}-\frac{1}{(n+2)^{1/2}} \leq  \frac{1}{(n+1)^{3/2}}.
\end{equation}
We can also find a constant $C_2>0$, independent of $n$ and $k$, such that
\begin{equation}\label{eq:est_frac_qnumbers}
\left\vert 1-\frac{q^{-1/2}[k]_q^{1/2}}{[k+1]_q^{1/2}} \right\vert= \left\vert 1-\frac{(1-q^{2k})^{1/2}}{(1-q^{2k+2})^{1/2}} \right \vert \leq C_2q^{2k}.    
\end{equation}
Then, from Lemma~\ref{lem:Fredholm} and inequalities \eqref{eq:invRoots} and \eqref{eq:est_frac_qnumbers}, we obtain that \eqref{eq:commutatorVL} is no greater than 
\[ \left( \frac{1}{n+2}+3\sum_{k=1}^{n+1} \left(\frac{1}{(n+1)^3}+\frac{C_2^2q^{4k}}{n+1}+\frac{C_1^2q^{2k}}{n+1}\right) \right) \cdot \|\xi\|^2,\] 
which implies \eqref{eq:Fredholm3}.   
\end{proof}

\begin{remark}
The estimates in the proof of Theorem \ref{thm:Fredholm} yield analytic properties of $\Delta$ in the context of symmetrically normed ideals. For instance, the Schatten $s$-ideals $\mathcal{L}^s$ ($s>0$) of compact operators $T$ with singular values $(s_n(T))_{n\in \mathbb N}\in \ell^s(\mathbb N)$, or the logarithmic integral ideals $\mathrm{Li}^r$ ($r\leq 1$) of compact operators $T$ satisfying $s_n(T)=O((\log n)^{-r})$, see \cite{Ge,GoMe}. 

Indeed, consider the dense $*$-subalgebras of $\mathcal{O}_P$ and $\mathcal{O}_{P^{\dagger}}$, defined as \[\mathbb O_P:=\mathbb C^*[s_i:1\leq i\leq m],\qquad \mathbb O_{P^{\dagger}}:=\mathbb C^*[t_i:1\leq i\leq m],\]
respectively. Further, let $0<t\leq 1$ and $0<q<1$ be the unique solutions to \[ t+t^{-1}=m,\qquad q+q^{-1}=\sum_{i=1}^{m} |a_i|^2.\] 
From \cite[Lemma 1.6]{HaNe21} one obtains that $\mathrm{dim} H_n = [n+1]_t$, and hence 
$$\mathrm{dim} H_n=
\begin{cases}
    n+1, & \text{if } m=2\\
    O(t^{-n}), & \text{if } m>2.
\end{cases}
$$
Combining Lemma \ref{lem:norm_est} with standard singular value estimates (see \cite[Lemma 7.4]{Ge}), we obtain a linear map $\eta:\mathbb O_P \odot \mathbb O_{P^{\dagger}}\to B(\Fock)$ that is a $*$-homomorphism modulo $\mathcal{L}^{s}$ for $s>\log_q t$, and $\tau(x)=\eta(x)+\mathcal{K}(\Fock)$ for $x\in \mathbb O_P \odot \mathbb O_{P^{\dagger}}.$ Thus, $\tau$ is $\mathcal{L}^{s}$-smooth on $\mathbb O_P \odot \mathbb O_{P^{\dagger}}$ for $s>\log_q t$, in the sense of Douglas \cite{Douglas} (see also \cite[Definition 2.6]{Ge}).

The corresponding notion for Fredholm modules, \emph{summability}, is more subtle, see for instance \cite{Ge,GoMe}. In a similar fashion, for $x\in \mathbb O_P \odot \mathbb O_{P^{\dagger}}$ we obtain that 
$$\eta(x)-V^*\pi(x)V\in
\begin{cases}
   \mathcal{L}^s, & \text{if } s>4,\,\, m=2,\\
   \mathrm{Li}^{1/2}, &\text{if }m>2.
\end{cases}
$$
Then, following \cite[Theorem 2.2.1]{GoMe} and given $x\in \mathbb O_P \odot \mathbb O_{P^{\dagger}}$, we observe that the commutators satisfy 
\[[F, \pi(x)]\in
\begin{cases}
   \mathcal{L}^s, & \text{if } s>8,\,\, m=2,\\
   \mathrm{Li}^{1/4}, &\text{if }m>2.
\end{cases}
\]
Therefore, $(\mathcal{H},\pi,F)$ is $s$-summable for $s>8$ when $m=2$, and $\mathrm{Li}^{1/4}$-summable when $m>2$. In particular, the Fredholm module $(\mathcal{H},\pi,F)$ lifts to a spectral triple, see \cite[Theorem 5]{GRU}.

This weakening of the regularity degree -- Schatten ideal to logarithmic integral ideal -- also shows up when passing from an extension to a Fredholm module representative of the fundamental class for Cuntz--Krieger algebras \cite{GoMe} and, more generally, in the context of Ruelle algebras~\cite{GWZ}. This phenomenon is related to the pure infiniteness of the latter C$^*$-algebras (see, for instance, \cite[Remark 2.10]{Ge}). However, the question whether the Cuntz--Pimsner algebras~$\mathcal{O}_P$ are typically purely infinite still stands. 

Finally, it would be interesting to understand if $\mathcal{O}_P$ admits a canonical groupoid model and if the smoothness-threshold $\log_q t$ of the extension $\tau$ holds any geometric information in relation to said model.
\end{remark}

We can now compute the slant product $-\otimes_{\mathcal{O}_{P}}\Delta\colon K_j(\mathcal{O}_{P})\to K^{j+1}(\mathcal{O}_{P^\dagger})$ in terms of Fredholm modules. This follows from general computations found in \cite[Proposition 4.5]{Ge} and \cite[Lemma 2.1.4]{GoMe}, hence we omit its proof. 

\begin{prop}\label{prop:slantprod} 
Under the assumptions of Theorem~\ref{thm:Fredholm}, the slant product $-\otimes_{\mathcal{O}_{P}}\Delta\colon K_0(\mathcal{O}_{P})\to K^{1}(\mathcal{O}_{P^\dagger})$ is given by $$[e]\mapsto [\mathcal{H}_e,\pi_e,F_e]\qquad (e\in M_k(\mathbb C)\otimes \mathcal{O}_P),$$ where  
\begin{itemize}
    \item[--] $\mathcal{H}_e=P_e(\mathbb C^k \otimes \mathcal{H})$, with $P_e$ being the projection $(id_{M_k(\mathbb C)}\otimes \pi)(e\otimes 1_{\mathcal{O}_{P^\dagger}});$
    \item[--] $\pi_e\colon \mathcal{O}_{P^\dagger}\to B(\mathcal{H}_e),\,t\mapsto (id_{M_k(\mathbb C)}\otimes \pi)(e\otimes t)$;
    \item[--] $F_e=P_e(id_{\mathbb C^k}\otimes F)P_e.$
\end{itemize}
On the other hand, the slant product $-\otimes_{\mathcal{O}_{P}}\Delta\colon K_1(\mathcal{O}_{P})\to K^{0}(\mathcal{O}_{P^\dagger})$ is given by $$[u]\mapsto \left[\mathcal{H}_k\oplus \mathcal{H}_k,\pi^{\dagger}_k\oplus \pi^{\dagger}_k, \begin{pmatrix}
    0 & G_{k,u}^*\\
    G_{k,u} & 0
\end{pmatrix}\right] \qquad (u\in M_k(\mathbb C)\otimes \mathcal{O}_P),$$
where 
\begin{itemize}
    \item[--] $\mathcal{H}_k=\mathbb C^k\otimes \mathcal{H}$;
    \item[--] $\pi^{\dagger}_k\colon \mathcal{O}_{P^\dagger}\to B(\mathcal{H}_k),\, t\mapsto id_{\mathbb C^k} \otimes \pi(1_{\mathcal{O}_P}\otimes t);$
    \item[--] $G_{k,u}=P_k(id_{M_k(\mathbb C)}\otimes \pi)(u\otimes 1_{\mathcal{O}_{P^\dagger}})P_k+1-P_k$, with $P_k$ being the projection $id_{\mathbb C^k}\otimes VV^*$. 
\end{itemize}
\end{prop}

As the embedding $\mathbb C \to \mathcal{T}_{P}$ is a $KK$-equivalence, the unit in $\OO_P$ defines a generator of~$K_0(\OO_P)$. Consequently, from the $KK$-duality and Proposition \ref{prop:slantprod} we obtain the following.

\begin{cor}
The group $K^{1}(\mathcal{O}_{P^\dagger})\cong \mathbb Z/(m-2)\mathbb Z$ is generated by $[\mathcal{H},1\otimes\rho^\dagger(-),F]$. In a similar way, $K^{1}(\mathcal{O}_{P})\cong \mathbb Z/(m-2)\mathbb Z$ is generated by $[\mathcal{H},\rho(-)\otimes1,F]$.
\end{cor}

\begin{remark}
The actions of $\GA$ on $\OO_P$ and $\OO_{P^\dagger}$ define unitary representations of $\GA$ on $L^2(\OO_P,\omega)$ and $L^2(\OO_{P^{\dagger}},\omega^\dagger)$, hence a unitary representation on $\HH$. It follows that the Fredholm module $(\mathcal{H},1\otimes\rho^\dagger(-),F)$ defines a class in the $\GA$-equivariant $K$-homology group $K^1_{\GA}(\OO_{P^\dagger})$. Similarly, $(\mathcal{H},\rho(-)\otimes1,F)$ defines a class in $K^1_{\GA}(\OO_{P})$.
\end{remark}

\section*{Declarations}
This work was (partly) supported by the Dutch Research Council (NWO) research programme VENI with project number 016.192.237, VI.Vidi.233.231, and by the NFR project 300837 ``Quantum Symmetry". The authors have no competing interests to declare that are relevant to the content of this article.

\section*{Data availability}
Our work follows a theoretical and mathematical approach, and as such does not generate any datasets. All our computations and conclusions are reproducible using the content of this work and of our reference list.

\end{document}